\documentclass[11pt]{article}
\usepackage{amsmath}
\usepackage{amsfonts} \makeatother
\usepackage{amsfonts,amsmath,amssymb,mathrsfs}
\usepackage{cite}
\usepackage{amsthm}
\usepackage[colorlinks,bookmarksopen,bookmarksnumbered,citecolor=red, linkcolor=red, urlcolor=red]{hyperref}

\numberwithin{equation}{section}
\newtheorem{theorem}{Theorem}[section]
\newtheorem{lemma}[theorem]{Lemma}
\newtheorem{proposition}[theorem]{Proposition}
\newtheorem{remark}{Remark}[section]

\newtheorem{corollary}[theorem]{Corollary}

\newcommand{\be}{\begin{equation}}
\newcommand{\ee}{\end{equation}}
\newcommand\bes{\begin{eqnarray}}
\newcommand\ees{\end{eqnarray}}
\newcommand{\bess}{\begin{eqnarray*}}
\newcommand{\eess}{\end{eqnarray*}}

\newcommand\eps{\varepsilon}

\def\R{\mathbb{R}}
\def\N{\mathbb{N}}

\begin{document}

\begin{center}{\bf\Large On a class of planar Schr\"{o}dinger-Poisson systems

                  with a bounded potential well
\footnote{This work was supported by the National Natural Science Foundation of China (No. 11901284),
the Natural Science Foundation of Jiangsu Province (No. BK20180814),
the China Postdoctoral Science Foundation funded project (No. 2020M671531)
and Jiangsu Planned Projects for Postdoctoral Research Funds (No. 2019K097).}
}\\[4mm]
 { \large Miao Du$^{\textrm{a,b},*}$, \ \ Jiaxin Xu$^{\textrm{a}}$}\\[4mm]
{\small
 $^{\textrm{a}}$  School of Applied Mathematics, Nanjing University of Finance and Economics,

     Nanjing 210023, PR China

 $^{\textrm{b}}$ School of Mathematical Sciences, Nanjing Normal University, Nanjing 210023, PR China
}
\renewcommand{\thefootnote}{}
\footnote{\hspace{-2ex}$^{\ast}$ Corresponding author}
\footnote{\hspace{-2ex}\emph {~E-mail addresses:}
 dumiaomath@163.com, xujiaxin991226@163.com.}
\end{center}

\setlength{\baselineskip}{16pt}

\begin{quote}
  \noindent {\bf Abstract:} In this paper, we deal with the planar Schr\"{o}dinger-Poisson system
  \begin{equation*}
    \begin{cases}
      -\Delta u + V(x) u + \phi u = b|u|^{p-2} u \ &\text{in}\ \R^{2},\\
      \Delta \phi= u^{2}   &\text{in}\ \R^{2},
    \end{cases}
  \end{equation*}
  where $b \geq 0$, $p > 2 $ and $V \in C(\R^2, \R)$ is a potential function with
  $\inf_{\R^2} V >0$. Suppose moreover that $V$ exhibits a bounded potential well
  in the sense that $\lim_{|x|\rightarrow \infty} V(x)$ exists and is equal to $\sup_{\R^2} V$.
  By using variational methods, we obtain the existence of ground state solutions
  for this system in the case where $p \geq 3$.
  Furthermore, we also present a minimax characterization of ground state solutions.
  The main feature of this work is that we do not assume any periodicity or symmetry condition
  on the external potential $V$, which is essential to establish the compactness condition of Cerami sequences.

  \noindent {\bf MSC}: 35J50; 35Q40

  \noindent {\bf Keywords}: {Planar Schr\"{o}dinger-Poisson system;
    bounded potential well;  ground state solution; variational method}
\end{quote}

\section{Introduction}

$ $
\indent
This paper is devoted to the existence of ground state solutions (i.e., least energy solutions) for the
following planar Schr\"{o}dinger-Poisson system with pure power nonlinearities:
\begin{equation}\label{eq 1.1}
  \begin{cases}
     -\Delta u + V(x) u + \phi u = b|u|^{p-2} u \ &\text{in}\ \R^{2},\\
     \Delta \phi= u^{2}   &\text{in}\ \mathbb{R}^{2},
  \end{cases}
\end{equation}
where $b \geq 0$, $p>2$ and the external potential $V$ satisfies the following condition:
\begin{itemize}
  \item [$(V_0)$] $V \in C(\R^2, \R)$ and
  $0< V_0:= \inf_{\R^2} V < \sup_{\R^2} V =\lim_{|x|\rightarrow \infty} V(x)=: V_\infty < \infty$.
\end{itemize}
This kind of hypotheses has been introduced to investigate various types of elliptic problems, and we merely
refer the reader to \cite{Rabinowitz-1992} for the study of the nonlinear  Schr\"{o}dinger equation.
Note that, the condition $(V_0)$ implies that the potential function $V$ exhibits a bounded potential well.
The existence of potential wells is more rigorous than local minima, but has the advantage in some situations,
see e.g. \cite{Liu-Zhang-2023} and the references therein.

The consideration of \eqref{eq 1.1} is inspired by the recent studies on the Schr\"{o}dinger-Poisson
system of the type
\begin{equation}\label{eq 1.2}
   \begin{cases}
      i\psi_{t}- \Delta \psi + W(x) \psi + \gamma \phi \psi = b |\psi|^{p-2}\psi
        \ & \text{in}\ \mathbb{R}^{N}\times \mathbb{R},\\
      \Delta \phi= |\psi|^{2}   & \text{in}\ \mathbb{R}^{N}\times \mathbb{R},
   \end{cases}
\end{equation}
where $N \geq 2$,  $\psi: \R^N \times \R \rightarrow \mathbb{C}$ is the time-dependent
wave function, $W: \R^N \rightarrow \R$ is a real external potential,
$\gamma \in \R$ is the coupling constant, $b \geq 0$  and $2< p <2^{*}$.
Here, $2^{\ast}$ is the so-called critical Sobolev exponent,
i.e., $2^{\ast}=\frac{2N}{N-2}$ if $N \geq 3$ and  $2^{\ast}=\infty$ if $N=1$ or $2$.
The function $\phi$ represents an internal potential for a nonlocal self-interaction of
the wave function $\psi$. System \eqref{eq 1.2} arises in many important problems from physics,
such as quantum mechanics (see e.g. \cite{Benguria-1981,Catto-Lions-1993})
and semiconductor theory (see e.g. \cite{Lions-1987,Markowich-1990}).
We refer the reader to \cite{Benci-Fortunato-1998,Mauser-2001}
for more physical backgrounds of \eqref{eq 1.2}.

One of the most interesting questions about \eqref{eq 1.2} is the existence
of standing wave solutions. The usual standing wave ansatz
$\psi(x,t)=e^{-i\lambda t}u(x)$, $\lambda\in \R$,
reduces \eqref{eq 1.2} to the system
\begin{equation}\label{eq 1.3}
  \begin{cases}
    -\Delta u + V(x) u + \gamma \phi u = b|u|^{p-2} u \ &\text{in}\ \R^{N},\\
    \Delta \phi= u^{2}   &\text{in}\ \R^{N},
  \end{cases}
\end{equation}
where $V(x)=W(x)+\lambda$. The second equation in \eqref{eq 1.3} determines
$\phi: \mathbb{R}^{N}\rightarrow\mathbb{R}$ only up to harmonic functions.
It is natural to choose $\phi$ as the negative Newton potential of $u^{2}$,
that is, the convolution of $u^{2}$ with the fundamental solution $\Phi_{N}$
of the Laplacian, which is expressed by
\begin{equation*}
  \Phi_{N}(x)=\frac{1}{2 \pi} \log |x| \quad \text{if} \ N=2
    \qquad \text{and} \qquad
  \Phi_{N}(x)= \frac{1}{N(2-N)\omega_{N}}|x|^{2-N} \quad  \text{if} \ N \geq 3.
\end{equation*}
Here, as usual, $\omega_{N}$ denotes the volume of the unit ball in $\R^{N}$.
With this formal inversion of the second equation in \eqref{eq 1.3},
we can deduce the integro-differential equation
\begin{equation}\label{eq 1.4}
   -\Delta u + V(x) u + \gamma \left(\Phi_{N}\ast|u|^{2}\right)u  = b|u|^{p-2}u
   \quad \text{in}\ \R^N.
\end{equation}
In the past few decades, equation \eqref{eq 1.4} and its generalizations have been
widely investigated and are quite well understood in the case $N \geq 3$.
The majority of the literature focuses on the study of \eqref{eq 1.4} with $N=3$ and $\gamma<0$.
In this case, existence, nonexistence and multiplicity results of solutions
have been obtained by using variational methods,  see e.g. \cite{Ambrosetti-Ruiz-2008,
Azzollini-Pomponio-2008,Bellazzini-Jeanjean-Luo-2008,Cerami-2010,DAprile-2004-1,DAprile-2004-2,
Ruiz-2006,Wangjun2012,Wang-Zhou-2007,Zhaoleiga-2013,Zhaoleiga-2008} and the references therein.

In contrast with the higher-dimensional case $N \geq 3$, variational approach
cannot be adapted straightforwardly to the planar case $N=2$ due to the fact that
the logarithmic convolution kernel $\Phi_2(x)=\frac{1}{2 \pi} \log |x|$
is sign-changing and presents singularities as $|x|$ goes to zero and infinity.
We remark that, at least formally, \eqref{eq 1.4} has a variational structure
related to the energy functional
\begin{equation*}
   I_N (u) := \frac{1}{2}\int_{\R^N}\left(|\nabla u|^{2}
     + V(x) u^{2}\right)\textrm{d}x + \frac{\gamma}{4}\int_{\R^N}\int_{\R^N}
     \Phi_N \left(|x-y|\right)u^{2}(x)u^2(y)\,\textrm{d}x\textrm{d}y
     - \frac{b}{p}\int_{\R^N} |u|^p\,\textrm{d}x,
\end{equation*}
whereas $I_2$ is not well-defined on the natural Hilbert space $H^{1}(\R^2)$
even if $V \in L^\infty(\R^2)$, and this is one of the reasons why much less is known
in the planar case in which \eqref{eq 1.4} becomes
\begin{equation}\label{eq 1.5}
  -\Delta u+ V(x)u + \frac{\gamma}{2 \pi}\left(\,\log\left(|\cdot|\right)\ast|u|^{2}\,\right)u
   = b |u|^{p-2}u  \quad \text{in}\ \R^{2}.
\end{equation}
To overcome this obstacle, Stubbe \cite{Stubbe-2008} introduced the smaller Hilbert space
\begin{equation*}
    X:=\left\{u \in H^{1}(\mathbb{R}^{2}):\: \int_{\mathbb{R}^{2}}
    \log\left(1+|x|\right)u^{2}\,\textrm{d}x < \infty \right\},
\end{equation*}
which ensures that the associated energy functional is well-defined and of class $C^1$ on $X$.
Considering the case $ V(x) \equiv  \lambda \geq 0$, $\gamma>0$ and $b=0$,
by using strict rearrangement inequalities he proved that \eqref{eq 1.5} has a unique ground state
which is a positive spherically symmetric decreasing function.
Later, Cingolani and Weth \cite{Cingolani-Weth-2016} developed some new ideas and estimates
within the underlying space $X$, and then detected the existence of ground states and high energy solutions
for \eqref{eq 1.5} with $\gamma>0$, $b\geq0$ and $p\geq4$ in a periodic setting.
The key tool in \cite{Cingolani-Weth-2016} is a strong compactness condition (modulo translation)
for Cerami sequences at arbitrary positive energy levels.
Such a property fails to hold in higher space dimensions, and it is also not available in the case where $2<p<4$.
Successively, Weth and the first author \cite{Du-Weth-2017} removed the restriction
$p \geq 4$ in \cite{Cingolani-Weth-2016}, and also obtained the existence of ground states
and high energy solutions for \eqref{eq 1.5} in the case where $V$ is a positive constant and $2<p<4$.
When $\gamma>0$ and $b=0$, equation \eqref{eq 1.5} is also known as the logarithmic Choquard equation
and can be derived from the Schr\"{o}dinger-Newton equation \cite{Penrose-1996}.
In \cite{Cingolani-Weth-2016}, it has been proved that the logarithmic Choquard equation
has a unique (up to translation) positive solution in the case where $V$ is a positive constant.
In \cite{Bonheure-2017}, Bonheure,  Cingolani and Van Schaftingen  showed the sharp asymptotics
and nondegeneracy of this unique positive solution.
For more related works on the planar Schr\"{o}dinger-Poisson system, see e.g.
\cite{Albuquerque-2021,Alves-Figueiredo-2019,Azzollini-2021,Cassani-2017,
Chen-2020-2,Chen-2020,Chen-2019,Cingolani-Jeanjean-2019,Liu-Zhang-2022,Liu-Zhang-2023-1} and the references therein.

At this moment, we would like to point out that in all the works mentioned above for the planar Schr\"{o}dinger-Poisson system,
the periodicity or symmetry assumption on the external potential $V$ plays a key role in recovering the compactness.
Until very recently, Molle and Sardilli \cite{Molle-Sardilli-2022} considered \eqref{eq 1.5} with $\gamma>0$, $b>0$
and $p\geq4$ in a nonperiodic and nonsymmetric setting, where $V$ satisfies
\begin{itemize}
  \item [$(A_1)$] $V \in L_{loc}^1(\R^2)$, $\inf_{\R^2} V>0$ and
    $\left|\{x \in \R^2:\: V(x)\leq M\}\right|<\infty$ for every $M>0$.
\end{itemize}
Note that, the embedding of $H_V:=\left\{u \in H^{1}(\R^2) :\: \int_{\R^2} V(x)u^{2}\,\textrm{d}x < \infty \right\}$
into $L^s(\R^2)$ is compact for all $s \geq 2$ (see \cite{Bartsch-Wang-1995}).
Therefore, the Cerami compactness condition holds at arbitrary positive energy levels by noticing that
the weak limit of the Cerami sequence in $H_V$ is not equal to zero.
Using a variant of the mountain pass theorem, the authors \cite{Molle-Sardilli-2022}
proved that \eqref{eq 1.5} has a positive ground state solution. Meanwhile, Liu, R\u{a}dulescu and Zhang
\cite{Liu-Zhang-2023} also studied the existence of positive ground state solutions for \eqref{eq 1.5}
when $\gamma>0$, $V\in C(\R^2, \R)$ satisfies
\begin{itemize}
  \item [$(A_2)$] $V$ is weakly differentiable,
    $(\nabla V(x), x) \in L^s (\R^2)$ for $s \in (1, \infty]$
    and $2V(x)+ (\nabla V(x),  x) \geq 0$ for a.e. $x \in \R^2$,
    where $(\cdot, \cdot)$ is the usual inner product in $\R^2$;

  \item [$(A_3)$] for all $x \in \R^2$, $V(x) \leq \lim_{|y|\rightarrow \infty} V(y):=V_\infty< \infty$
    and the inequality is strict in a subset of positive Lebesgue measure;

  \item [$(A_4)$] $\inf \sigma\left(-\Delta +V(x)\right)>0$, where $\sigma\left(-\Delta +V(x)\right)$ denotes
    the spectrum of the self-adjoint operator $-\Delta +V(x):\: H^1(\R^2) \rightarrow L^2(\R^2)$, that is,
    \begin{equation*}
       \inf \sigma \bigl(-\Delta +V(x)\bigr) = \inf_{u \in H^1(\R^2) \backslash \{0\}}
       \frac{\int_{\R^2}\left(|\nabla u|^{2} + V(x) u^{2}\right)\textrm{d}x}{\int_{\R^2} u^{2}\:\textrm{d}x}>0,
    \end{equation*}
\end{itemize}
and $b|u|^{p-2}u$ is replaced by $f(u)$ which is required to have either a subcritical
or a critical exponential growth in the sense of Trudinger-Moser. However,
we observe that their results do not cover some representative cases, such as
the pure power nonlinearity $|u|^{p-2}u$, because of their assumptions $(f_0)$ and $(f_4)$.
A natural question for us is whether there exist ground state solutions for \eqref{eq 1.5}
with a bounded potential. As far as we know, no existence results for \eqref{eq 1.5}
have been available for this case. This is the basic motivation of the present work.

In this paper, we focus on \eqref{eq 1.5} in the case $\gamma>0$, and by rescaling
we may assume that $\gamma = 1$. More precisely, we are dealing with system \eqref{eq 1.1},
the associated scalar equation
\begin{equation}\label{eq 1.6}
    -\Delta u + V(x)u + \frac{1}{2\pi} \left(\log\left(|\cdot|\right)\ast|u|^{2}\right)u
     = b|u|^{p-2}u  \quad  \text{in}\ \R^{2}
\end{equation}
and the associated energy functional $I: X \to \R$ defined by
\begin{equation}\label{eq 1.7}
 I(u)=\frac{1}{2}\int_{\mathbb{R}^{2}}\left(|\nabla u|^{2}
     +V(x)u^{2}\right)\textrm{d}x + \frac{1}{8 \pi}\int_{\mathbb{R}^{2}}
     \int_{\mathbb{R}^{2}}\log \left(|x-y|\right)u^{2}(x)u^2(y)\,\textrm{d}x\textrm{d}y
     -\frac{b}{p} \int_{\mathbb{R}^{2}}|u|^p \,\textrm{d}x.
\end{equation}
In the following, by a solution of \eqref{eq 1.6} we always mean a weak solution,
i.e., a critical point of $I$. A nontrivial solution $u$ of \eqref{eq 1.6}
is called a ground state solution if $I(u) \leq I(w)$ for any nontrivial solution
$w$ of \eqref{eq 1.6}. The main aim of this paper is to obtain the existence
of ground state solutions for \eqref{eq 1.6} with a bounded potential well.
Additionally, we also present a minimax characterization of ground state solutions.

Our first main result is concerned with the existence of ground state solutions
for \eqref{eq 1.6} in the case where $p \geq 4$. For this we define
the Nehari manifold associated to the functional $I$ by
\begin{equation}\label{eq 1.8}
  \mathcal{N} = \left\{ u \in X \backslash \{0\}: \: I'(u) u = 0\right\}.
\end{equation}

\begin{theorem}\label{th 1.1}
Suppose that $b \geq 0$, $p \geq 4$, and that $(V_0)$ holds.
Then the restriction of $I$ to $\mathcal{N}$ attains a global minimum,
and every minimizer $\bar{u} \in \mathcal{N}$ of $I|_{\mathcal{N}}$
is a solution of \eqref{eq 1.6} which does not change sign
and obeys the minimax characterization
\begin{equation*}
  I(\bar{u}) = \inf_{u \in X \backslash \{0\}} \sup_{t>0} I(tu).
\end{equation*}
\end{theorem}

\begin{remark}\rm
Theorem \ref{th 1.1} implies that \eqref{eq 1.6} has a ground state solution in $X$, and every
ground state solution of \eqref{eq 1.6} does not change sign and obeys a simple minimax characterization.
Note that, the hypothesis $\inf_{\R^2} V >0 $ in $(V_0)$ can be weakened to $(A_4)$.
As a consequence, the conclusions of Theorem \ref{th 1.1} still hold in the case where $V \in C(\R^2, \,\R)$
satisfies $(A_3)$ and $(A_4)$. Compared with \cite{Liu-Zhang-2023},
the assumption $(A_2)$ on the potential $V$ is removed and a minimax characterization of
ground state solutions for \eqref{eq 1.6} is also provided in Theorem \ref{th 1.1}.
\end{remark}

To prove Theorem \ref{th 1.1}, we shall use the method of Nehari manifold as e.g.
in \cite{Badiale-Serra,Rabinowitz-1992,Szulkin-Weth,Willem-1996}. Traditionally, this is done in three steps.
In the first step, we show that the infimum of $I$ on $\mathcal{N}$  is greater than zero.
In the second step, we prove that the infimum of $I$ on $\mathcal{N}$ can be attained.
In the third step, we show that every minimizer of $I|_{\mathcal{N}}$ is a critical point of $I$.
Note that, the first and third steps are somewhat standard and the main difficulties often lie in the second step.
We now sketch the main idea of proving the second step as follows: First, it is easy to verify that
every minimizing sequence $\{u_n\}$ for $I|_{\mathcal{N}}$ is bounded in $H^1(\R^2)$ and this implies that,
up to a subsequence, there exists $u \in H^1(\R^2)$ such that $u_n \rightharpoonup u$ in $H^1(\R^2)$.
Then, we show that $u \neq 0$, which is the key difficulty. To overcome this obstacle, we need to
consider the associated limit equation of \eqref{eq 1.6} in which $V(x)$ is replaced by $V_\infty$,
the corresponding limit functional $I_\infty$ and Nehari manifold $\mathcal{N}_\infty$.
Suppose by contradiction that $u =0$. By the assumption on the asymptotic shape of $V$,
a delicate analysis gives $\inf_{\mathcal{N}} I \geq \inf_{\mathcal{N}_\infty} I_\infty$.
On the other hand, since $V_\infty >0$ is a constant, it follows from \cite[Theorem 1.1]{Cingolani-Weth-2016}
and $(V_0)$ that $\inf_{\mathcal{N}} I < \inf_{\mathcal{N}_\infty} I_\infty$, a contradiction.
Finally, following the argument in \cite{Cingolani-Weth-2016}, we find that
$\{u_n\}$ is bounded in $X$, and so, after passing to a subsequence again, $u_n \rightharpoonup u$ in $X$.
By the weak lower semicontinuity of the norm and the compact embedding of $X$, we derive that
the infimum of $I$ on $\mathcal{N}$ can be attained in a standard way.

It is worth noticing that Theorem \ref{th 1.1} fails to hold in the case where $2<p<4$.
Besides the lack of compactness, the key sticking point in this case is the competing nature of the local
and nonlocal superquadratic terms in the functional $I$.
In particular, we note that the nonlinearity $u \mapsto f(u):=|u|^{p-2}u$ with $2<p<4$ does not satisfy
the Ambrosetti-Rabinowitz type condition
\begin{equation*}
  0<\mu\int_{0}^{u} f(s)\,\textrm{d}s \leq f(u)u \qquad \text{for all $u\neq0$ with some $\mu>4$},
\end{equation*}
which obviously implies that Palais-Smale sequences or Cerami sequences are bounded in $H^1 (\R^2)$.
Moreover, the fact that the function $f(s)/|s|^{3}$ is not increasing on $(-\infty, 0)$ and $(0, \infty)$ prevents
us from using the method of Nehari manifold.

Our second main result is concerned with the existence of ground state solutions for \eqref{eq 1.6}
in the case where $3 \leq p < 4$. In this case, except for $(V_0)$, we need the following condition: 
\begin{itemize}
  \item [$(V_1)$] $V \in C^1(\R^2, \R)$ and there exists $\eta > 0$ such that
  $\left|(\nabla V(x), x)\right| \leq  \eta$ for all $x \in \R^2$.
\end{itemize}
This condition is used to construct a Cerami sequence with a key additional property, from which we can
easily conclude that this Cerami sequence is bounded in $H^1 (\R^2)$.

\begin{theorem}\label{th 1.2}
Suppose that $b \geq0$, $p \geq 3$, and that $(V_0)$ and $(V_1)$ hold. Then \eqref{eq 1.6}
has a ground state solution in $X$.
\end{theorem}

\begin{remark} \rm
Theorem \ref{th 1.2} shows, in particular, that the existence of ground state solutions for \eqref{eq 1.6}
with $b>0$  and $3\leq  p < 4$.
Since we do not know whether the mountain pass energy coincides with the ground state energy and
there exists a saddle point structure for the limit functional $I_\infty$ (see \cite[p. 3496]{Du-Weth-2017}),
the argument in Theorem \ref{th 1.2} is not available in the case where $2< p <3$, see Proposition \ref{prop 4.4}
and Lemma \ref{lem 4.5} below for more details. Consequently, it still remains an open problem whether \eqref{eq 1.6}
has a ground state solution in $X$ for the case $2< p <3$.
\end{remark}

Our proof of Theorem \ref{th 1.2} is inspired by \cite{Du-Weth-2017,Jeanjean-1997}.
First we construct a Cerami sequence $\{u_{n}\} \subset X$ at the mountain pass level with an extra property.
By this extra information, we can deduce the boundedness of $\{u_{n}\}$ in $H^{1}(\R^2)$.
Then following the argument in \cite{Cingolani-Weth-2016}, a careful analysis on the mountain pass levels
corresponding to the functional $I$ and the associated limit functional $I_\infty$ shows that,
after passing to a subsequence, $\{u_{n}\}$ converges to a mountain pass solution of \eqref{eq 1.6}.
Therefore, the set $\mathcal K$ of nontrivial solutions of \eqref{eq 1.6} is nonempty. Finally
we consider a sequence $\{u_{n}\} \subset \mathcal K$ such that $I(u_n) \rightarrow \inf_{\mathcal K} I$
as $n \to \infty$, and in the same way as before we may pass to a subsequence which converges to
a ground state solution of \eqref{eq 1.6}.

Our third main result is concerned with the minimax characterization of ground state solutions
for \eqref{eq 1.6} in the case where $3 \leq p < 4$. For this purpose, besides $(V_0)$,
we need to add the following conditions:
\begin{itemize}
  \item [$(V_2)$] $V \in C^1(\R^2, \R)$ and if $\mathcal{V}(x) := V(x)-\frac{1}{2}(\nabla V(x),x)$,
   then the function $(0, \infty) \rightarrow \R$, $t \mapsto \mathcal{V}(tx)$ is nondecreasing
   on $(0, \infty)$ for every $x \in \R^2$; 

  \item [$(V_3)$] $V(x) + \frac{1}{2} (\nabla V(x),x) \leq V_\infty$ for all $x \in \R^2$.
\end{itemize}
There are indeed many functions which satisfy $(V_0)$ and $(V_1)$$-$$(V_3)$.
Here we present two examples. One example is given by $V(x)= 1 - \frac{1}{2+|x|^2}$,
and the other is given by $V(x)= 1- \frac{1}{2+ \log(1+|x|^2)}$.

Inspired by \cite{Du-Weth-2017}, we now define the auxiliary functional $J: X \to \R$  by
\begin{align}\label{eq 1.9}
  J(u)=&\int_{\R^2} \left(2 |\nabla u|^2+  \mathcal{V}(x) u^2 -\frac{2b(p-1)}{p}|u|^{p}\right)\textrm{d}x
     -\frac{1}{8 \pi} \left(\int_{\mathbb{R}^{2}} u^{2}\,\textrm{d}x\right)^{2} \nonumber \\
   &+ \frac{1}{2 \pi}\int_{\R^2}\int_{\R^2}\log\left(|x-y|\right)u^{2}(x)u^{2}(y)\,\textrm{d}x\textrm{d}y,
\end{align}
and set
\begin{equation}\label{eq 1.10}
  \mathcal M:= \{u \in X \setminus \{0\}:\: J(u)=0\}.
\end{equation}
It then follows in a standard way from a Pohozaev type identity given in Lemma~\ref{lem 2.4} below
that every nontrivial solution of \eqref{eq 1.6} is contained in $\mathcal M$.
In the sequel, we call the set $\mathcal{M}$ the Nehari-Pohozaev mainfold.
It has been proposed by Ruiz in \cite{Ruiz-2006} for the study of \eqref{eq 1.4} with $N=3$.

\begin{theorem}\label{th 1.3}
Suppose that $b \geq0 $, $p \geq 3$, and that $(V_0)$, $(V_2)$ and $(V_3)$ hold.
Then the restriction of $I$ to $\mathcal{M}$ attains a global minimum,
and every minimizer $\bar{u} \in \mathcal{M}$ of $I|_{\mathcal{M}}$ is a solution
of \eqref{eq 1.6} which does not change sign and obeys the minimax characterization
\begin{equation*}
  I(\bar{u}) = \inf_{u \in X \backslash \{0\}} \sup_{t>0} I(u_t),
\end{equation*}
where $u_t \in X$ is defined by $u_t(x):= t^2u(tx)$ for $u \in X$ and $t > 0$.
\end{theorem}

\begin{remark} \rm
Theorem \ref{th 1.3} yields that \eqref{eq 1.6} has a ground state solution in $X$, and every
ground state solution of \eqref{eq 1.6} does not change sign and obeys a new minimax characterization.
However, this minimax characterization is lost for the case $2<p <3$, and we could not find any similar
saddle point structure of $I$.
\end{remark}

\begin{remark} \rm
By $(V_0)$ and $(V_2)$, we find that $(\nabla V(x),x) \geq 0$ for all $x \in \R^2$,
see Lemma \ref{lem 5.1} below. Therefore, using $(V_3)$  we further have
\begin{equation}\label{eq 1.11}
  \lim_{|x|\rightarrow\infty} (\nabla V(x),x) = 0 \qquad \text{and} \qquad
  V_0 \leq \mathcal{V}(x) \leq  V_\infty \quad \text{for all $x \in \R^2$,}
\end{equation}
so that $(V_1)$ follows. The condition $(V_2)$ is thought of as a monotonicity condition on the potential $V$,
which guarantees the saddle point structure of $I$ with respect to the fibres $\{u_t :\: t>0\} \subset X$,
$u \in X \setminus \{0\}$, see Lemma \ref{lem 5.3} below.
The condition $(V_3)$ is of key importance in the connection between the functional
$J$ and the associated limit functional $J_\infty$, and has been successfully applied
to study asymptotically linear Schr\"{o}dinger equations in \cite{Lehrer-Maia-2014}.
\end{remark}

This paper is organized as follows. In Section \ref{sec 2}, we set up the variational
framework for \eqref{eq 1.6} and present some useful preliminary results.
In Section \ref{sec 3}, we give the proof of Theorem \ref{th 1.1}
on the existence of ground state solutions to \eqref{eq 1.6} for the case $p \geq 4$.
Section \ref{sec 4} is devoted to the proof of Theorem \ref{th 1.2} on the existence of ground state solutions
to \eqref{eq 1.6} for the case $3\leq p < 4$. Finally, in Section \ref{sec 5} we complete
the proof of Theorem \ref{th 1.3}.

Throughout the paper,  we shall make use of the following notation.
$L^{s}(\mathbb{R}^{2})$ denotes the usual Lebesgue space with the norm $|\cdot|_{s}$ for $1\leq s \leq \infty$.
For any $z \in \mathbb{R}^{2}$ and for any $\rho>0$, $B_{\rho}(z)$ denotes the ball of radius $\rho$ centered at $z$.
$X'$ stands for the dual space of $X$.
As usual, the letters $C$, $C_{1}$, $C_{2}, \cdot\cdot\cdot$ denote positive constants
that can change from line to line. Finally, when taking limits, the symbol $o(1)$ stands for any quantity
that tends to zero.

\section{Preliminaries}\label{sec 2}

$ $
\indent
In this section, we review the variational setting for \eqref{eq 1.6} as elaborated
by \cite{Cingolani-Weth-2016} and present some useful preliminary results.
In the following, we always assume that $b \geq 0$, $p > 2$, and that $(V_0)$ holds.
Let $H^{1}(\mathbb{R}^{2})$ be the usual Sobolev space endowed with
the scalar product and norm
\begin{equation*}
  \langle u, v\rangle =\int_{\mathbb{R}^{2}} \left(\nabla u \cdot \nabla v
     + V(x)uv \right)\textrm{d}x \qquad \text{and}  \qquad \|u\| =\langle u, u\rangle^{1/2}.
\end{equation*}
Thanks to $(V_0)$, these are equivalent to the standard scalar product
and norm of $H^1(\R^2)$.
We now define, for any measurable function $u: \, \mathbb{R}^{2}\rightarrow \mathbb{R}$,
\begin{equation*}
  |u|_{*} =  \left(\int_{\mathbb{R}^{2}} \log\left(1+|x|\right)u^{2}\,\textrm{d}x\right)^{1/2} \in [0, \infty].
\end{equation*}
As already noted in the introduction, we shall work in the Hilbert space
\begin{equation*}
  X=\left\{u \in H^{1}(\mathbb{R}^{2}):\: |u|_{*}< \infty\right\}
\end{equation*}
with the norm given by $\|u\|_{X}:= \sqrt{\|u\|^{2}+|u|^{2}_{*}}$.
Next, we define the symmetric bilinear forms
\begin{align*}
  (u, v) \mapsto B_{1}(u, v) &= \frac{1}{2 \pi} \int_{\mathbb{R}^{2}}\int_{\mathbb{R}^{2}}
     \log \left(1+|x-y|\right)u(x)v(y)\,\textrm{d}x\textrm{d}y,\\
  (u, v) \mapsto B_{2}(u, v) &= \frac{1}{2 \pi} \int_{\mathbb{R}^{2}}\int_{\mathbb{R}^{2}}
     \log \left(1+\frac{1}{|x-y|}\right)u(x)v(y)\,\textrm{d}x\textrm{d}y,\\
  (u, v) \mapsto B_{0}(u, v) &= B_{1}(u, v)-B_{2}(u, v) = \frac{1}{2\pi}\int_{\mathbb{R}^{2}}
     \int_{\mathbb{R}^{2}}\log \left(|x-y|\right)u(x)v(y)\,\textrm{d}x\textrm{d}y,
\end{align*}
where in each case, the definition is restricted to measurable functions
$u,v : \R^{2}\rightarrow \R $ such that the corresponding double integral
is well-defined in the Lebesgue sense. Then we define on $X$ the associated functionals
\begin{align*}
  N_{1}(u)& :=B_{1}\left(u^{2}, u^{2}\right)= \frac{1}{2 \pi}
     \int_{\mathbb{R}^{2}}\int_{\mathbb{R}^{2}}\log \left(1+|x-y|\right)u^{2}(x)u^{2}(y)\,\textrm{d}x\textrm{d}y,\\
  N_{2}(u)& :=B_{2}\left(u^{2}, u^{2}\right)= \frac{1}{2 \pi} \int_{\mathbb{R}^{2}}
     \int_{\mathbb{R}^{2}}\log \left(1+\frac{1}{|x-y|}\right)u^{2}(x)u^{2}(y)\,\textrm{d}x\textrm{d}y,\\
  N_{0}(u)& :=B_{0}\left(u^{2}, u^{2}\right)= \frac{1}{2 \pi}\int_{\mathbb{R}^{2}}
     \int_{\mathbb{R}^{2}}\log \left(|x-y|\right)u^{2}(x)u^{2}(y)\,\textrm{d}x\textrm{d}y.
\end{align*}
For notational convenience, we rewrite the functional $I(u)$ defined by \eqref{eq 1.7}
in the following form:
\begin{equation*}
  I(u)= \frac{1}{2}\|u\|^2+ \frac{1}{4}N_0(u)-\frac{b}{p}|u|^{p}_{p}.
\end{equation*}
Since
\begin{equation*}
  \log\left(1+|x-y|\right) \leq \log\left(1+|x|+|y|\right)
  \leq \log\left(1+|x|\right) + \log\left(1+|y|\right)
  \quad \text{for}\ x, y \in \R^2,
\end{equation*}
we have the estimate
\begin{align*}
  B_{1}(uv, wz) &\leq \frac{1}{2 \pi}\int_{\mathbb{R}^{2}}\int_{\mathbb{R}^{2}}
  \bigl[\log\left(1+|x|\right)+\log\left(1+|y|\right)\bigr]
  \left|u(x)v(x)\right|\left|w(y)z(y)\right|\textrm{d}x\textrm{d}y \notag\\
  &\leq \frac{1}{2 \pi} \left(|u|_{*}|v|_{*}|w|_{2}|z|_{2}
    + |u|_{2}|v|_{2}|w|_{*}|z|_{*}\right)
  \quad \text{for}\ u, v, w, z \in X.
\end{align*}
Note that $0< \log (1+r)<r$ for $r>0$, it then follows from the Hardy-Littlewood-Sobolev inequality
(see \cite[Therorem 4.3]{Lieb-Loss-2001}) that
\begin{equation}\label{eq 2.1}
  \left|B_{2}(u,v)\right| \leq  \frac{1}{2 \pi} \int_{\mathbb{R}^{2}}
   \int_{\mathbb{R}^{2}}\frac{1}{|x-y|}\left|u(x)v(y)\right|\textrm{d}x\textrm{d}y
   \leq C_{0}|u|_{\frac{4}{3}}|v|_{\frac{4}{3}}  \quad
   \text{for}\ u, v \in L^{\frac{4}{3}}(\R^2)
\end{equation}
with a constant $C_{0}>0$, which readily implies that
\begin{equation}\label{eq 2.2}
    \left|N_{2}(u)\right| \leq C_{0} |u|_{\frac{8}{3}}^{4} \qquad
    \text{for}\ u \in L^{\frac{8}{3}}(\R^2).
\end{equation}

We need the following results from \cite{Cingolani-Weth-2016}.
\begin{lemma}[{\hspace{-0.05ex}\cite[Lemma 2.2]{Cingolani-Weth-2016}}]\label{lem 2.1}
The following properties hold true.

\begin{itemize}
  \item [\rm(i)] The space $X$ is compactly embedded in $L^{s}(\R^2)$
    for all $s \in [2, \infty)$.

  \item [\rm(ii)]  The functionals $N_{0}$, $N_{1}$, $N_{2}$ and $I$ are of class
    $C^{1}$ on $X$. Moreover,
     \begin{equation*}
        N'_{i}(u)v=4B_{i}(u^{2}, uv)\quad \text{for}\ u, v \in X \ \text{and}\  i=0, 1,  2.
     \end{equation*}

  \item [\rm(iii)] $N_{2}$ is continuously differentiable on $L^{\frac{8}{3}}(\mathbb{R}^{2})$.

  \item [\rm(iv)] $N_{1}$ is weakly lower semicontinuous on $H^{1}(\mathbb{R}^{2})$.

\end{itemize}
\end{lemma}

\begin{lemma}[{\hspace{-0.6ex} \cite[Lemma 2.1]{Cingolani-Weth-2016}}] \label{lem 2.2}
Let $\{u_{n}\}$ be a sequence in $L^{2}(\mathbb{R}^{2})$
such that $u_{n}\rightarrow u\in L^{2}(\mathbb{R}^{2})\backslash\{0\}$
pointwise a.e. in $\mathbb{R}^{2}$. Moreover, let $\{v_{n}\}$ be
a bounded sequence in $L^{2}(\mathbb{R}^{2})$ such that
\begin{equation*}
    \sup_{n \in \mathbb{N}} B_{1}\left(u_{n}^{2}, v_{n}^{2}\right) < \infty.
\end{equation*}
\vskip -0.1 true cm\noindent
Then there exist $n_{0} \in \mathbb{N}$ and $C>0$ such that
$|v_{n}|_{\ast}<C$ for $n \geq n_{0}$.
If moreover $B_{1}(u_{n}^{2}, v_{n}^{2}) \rightarrow 0$ and $|v_{n}|_{2} \rightarrow 0$
as $n \rightarrow \infty$, then $|v_{n}|_{\ast} \rightarrow 0$ as $n \rightarrow \infty$.
\end{lemma}

\begin{lemma}[{\hspace{-0.8ex} \cite[Lemma 2.6]{Cingolani-Weth-2016}}]\label{lem 2.3}
Let $\{u_{n}\}$, $\{v_{n}\}$ and $\{w_{n}\}$ be bounded sequences
in $X$ such that $u_{n} \rightharpoonup u$ weakly in $X$. Then, for every
$z \in X$, we have $B_{1}(v_{n}w_{n}, z(u_{n}-u)) \rightarrow 0$ as
$n\rightarrow\infty$.
\end{lemma}

Next, we provide a Pohozaev type identity for equation \eqref{eq 1.6}.

\begin{lemma}\label{lem 2.4}
Suppose that $b \geq 0$, $p>2$, and that $(V_0)$ and $(V_1)$ hold.
Let $u \in X$ be a weak solution of \eqref{eq 1.6}, then we have the following Pohozaev type identity:
\begin{equation*}
 P(u) := \int_{\R^2} \Bigl[ V(x)+ \frac{1}{2}(\nabla V(x), x) \Bigr]u^2 \,{\rm d}x
  + N_0(u) + \frac{1}{8 \pi} |u|_2^4 - \frac{2b}{p} |u|_p^p = 0.
\end{equation*}
\end{lemma}
\begin{proof}
The proof is standard, so we omit it here, see \cite[Lemma 2.4]{Du-Weth-2017} for a similar argument.
\end{proof}

We close this section with some observations on the functional geometry of $I$.

\begin{lemma}\label{lem 2.5}
There exists $\rho>0$ such that
\begin{equation}\label{eq 2.3}
  m_{\beta} := \inf \left\{I(u):\: u \in X, \ \|u\|=\beta\right\} >0
    \qquad \text{for}\ 0<\beta \leq \rho
\end{equation}
and
\begin{equation}\label{eq 2.4}
  n_{\beta} :=\inf \left\{I'(u)u:\: u \in X, \ \|u\|=\beta \right\}>0
    \qquad \text{for}\ 0<\beta \leq \rho.
\end{equation}
\end{lemma}

\begin{proof}
For each $u \in X$, by \eqref{eq 2.2} and the Sobolev embeddings we have
\begin{equation*}
  I(u)\geq \frac{1}{2}\|u\|^{2} -\frac{1}{4}N_{2}(u) - \frac{b}{p} |u|_p^p
    \geq \frac{1}{2}\|u\|^{2} - \frac{C_{0}}{4}|u|_{\frac{8}{3}}^{4}
       -\frac{b}{p}|u|_{p}^{p} \geq \frac{\|u\|^{2}}{2}\left(1-C_{1}\|u\|^{2}
       -C_{2}\|u\|^{p-2}\right),
\end{equation*}
which means that \eqref{eq 2.3} holds for $\rho>0$ sufficiently small. Since
\begin{equation*}
  I'(u)u =  \|u\|^2 + N_0(u) - b|u|_p^p \ge  \|u\|^2- N_2(u) - b|u|_p^p
   \qquad \text{for}\ u \in X,
\end{equation*}
a similar estimate indicates that \eqref{eq 2.4} holds for $\rho>0$ sufficiently small.
This ends the proof.
\end{proof}

\begin{lemma}\label{lem 2.6}
We have, for any $u \in X \backslash \{0\}$,
\begin{equation*}
  I(u_t) \to -\infty \qquad \text{as $t \to \infty$.}
\end{equation*}
In particular, the functional $I$ is not bounded from below.
\end{lemma}

\begin{proof}
Let $u\in X \backslash \{0\}$, then we have
\begin{equation*}
  I(u_{t}) \leq \frac{t^{4}}{2} |\nabla u|_2^2 +\frac{t^{2}}{2} V_\infty |u|_2^2
    + \frac{t^{4}}{4} N_0(u) -\frac{t^{4}\log t}{8 \pi}|u|_2^{4} -\frac{bt^{2p-2}}{p}|u|_p^{p}.
\end{equation*}
Therefore, $I(u_{t})\rightarrow -\infty$ as $t\rightarrow\infty$, and the claim follows.
\end{proof}

\section{Proof of Theorem \ref{th 1.1}} \label{sec 3}

$ $
\indent
In this section, we will prove Theorem \ref{th 1.1} on the existence of
ground state solutions for \eqref{eq 1.6} in the case where $p\geq 4$.
In the following, we always assume that $b \geq 0$, $p \geq 4$, and that $(V_0)$ holds.
To seek a ground state solution of \eqref{eq 1.6}, we consider the Nehari manifold
$\mathcal{N}$ defined in \eqref{eq 1.8}, i.e.,
\begin{equation*}
  \mathcal{N} = \left\{ u \in X \backslash \{0\} : \: I'(u) u=0\right\}
  =\left\{u \in X \backslash \{0\} : \: \|u\|^{2} + N_{0}(u) = b|u|^{p}_{p}\right\}.
\end{equation*}
It is easy to see that every nontrivial critical point of $I$ is contained in $\mathcal{N}$.
If $u\in \mathcal{N}$, then
\begin{equation*}
  I(u) = \frac{1}{4}\|u\|^{2} + \left(\frac{b}{4} - \frac{b}{p}\right)|u|^{p}_{p},
\end{equation*}
and since $p \geq 4$, it holds that
\begin{equation}\label{eq 3.1}
  I(u) \geq \frac{1}{4}\|u\|^2 >0 \qquad \text{for $u \in \mathcal{N}$}.
\end{equation}
We now define
\begin{equation*}
  m = \inf_{u \in \mathcal{N}}I(u),
\end{equation*}
and we try to show that $m$ is attained by some $u \in \mathcal{N}$ which is
a critical point of $I$ in $X$, and thus a ground state solution of \eqref{eq 1.6}.

To begin with, we present some basic properties of $\mathcal{N}$ and $I$.

\begin{lemma}\label{lem 3.1}
Let $u\in X \setminus\{0\}$, then the function $h_{u}: (0, \infty) \rightarrow \mathbb{R}$,
$h_{u}(t)=I(tu)$ is even and has the following properties.
\begin{enumerate}
  \item[\rm(i)] If
    \begin{equation}\label{eq 3.2}
      N_{0}(u)-b|u|^4_4<0 \quad \text{in case $p=4$} \qquad \text{and} \qquad
      N_{0}(u)<0 \ \, \text{or} \  \, b>0 \quad \text{in case $p>4$},
    \end{equation}
  then there exists a unique $t_u \in (0, \infty)$ such that $t_{u}u \in \mathcal{N}$
  and $I(t_{u}u)= \max\limits_{t > 0} I(tu)$. Moreover, $h'_u(t) > 0$ on $(0,  t_u)$
  and $h'_u(t) < 0$ on $(t_u, \infty)$, and $h_{u}(t) \rightarrow -\infty$ as $t \rightarrow \infty$.

  \item[\rm(ii)] If \eqref{eq 3.2} does not hold, then $h'_u(t) >0$ on $(0, \infty)$,
  and $h_{u}(t)\rightarrow \infty$ as $t\rightarrow \infty$.
\end{enumerate}
\end{lemma}

\begin{proof}
Observe that
\begin{equation*}
  \frac{h'_{u}(t)}{t}=\|u\|^2 + t^2 N_0(u) - bt^{p-2}|u|^p_p \qquad \text{for $t>0$},
\end{equation*}
the desired assertions follow easily.
\end{proof}

From Lemma \ref{lem 3.1}, we immediately deduce the following corollary.

\begin{corollary}\label{coro 3.2}
The Nehari manifold $\mathcal{N}$ is not empty and the infimum of $I$ on $\mathcal{N}$
obeys the following minimax characterization:
\begin{equation*}
  \inf_{u\in \mathcal{N}} I(u) = \inf_{u\in X\backslash \{0\}} \sup_{t>0} I(tu).
\end{equation*}
\end{corollary}

In the following lemma, we shall show that $m>0$.

\begin{lemma}\label{lem 3.3}
There results $m > 0$.
\end{lemma}

\begin{proof}
For any $u\in \mathcal{N}$, by \eqref{eq 2.2} and the Sobolev inequalities we have
\begin{equation*}
  \|u\|^2 = b|u|^p_p - N_0(u) \leq b|u|^p_p + N_2(u)
  \leq C_1\|u\|^p + C_2\|u\|^4.
\end{equation*}
Since $u \neq 0$ and $p>2$, we obtain $\alpha := \inf_{u\in \mathcal{N}} \|u\|^2 >0$.
It then follows from \eqref{eq 3.1} that
\begin{equation}\label{eq 3.3}
  m = \inf_{u \in \mathcal{N}}I(u) \geq \frac{1}{4} \inf_{u\in \mathcal{N}} \|u\|^2
    = \frac{1}{4} \alpha >0,
\end{equation}
as claimed.
\end{proof}

Next, we will prove that $m>0$ is achieved. For this purpose, we need to consider the associated
limit equation of \eqref{eq 1.6}, which is given as
\begin{equation}\label{eq 3.4}
  -\Delta u + V_\infty u + \frac{1}{2\pi} \left(\log\left(|\cdot|\right)\ast|u|^{2}\right)u
     = b|u|^{p-2}u  \quad  \text{in}\ \R^{2},
\end{equation}
where $V_\infty>0$ is defined in $(V_0)$. First we introduce in $H^1(\R^2)$ the scalar product and norm
\begin{equation*}
  \langle u, v\rangle_\star :=\int_{\mathbb{R}^{2}} \left(\nabla u \cdot \nabla v
     + V_\infty uv \right)\textrm{d}x \qquad \text{and}  \qquad \|u\|_\star :=\langle u, u\rangle^{1/2}.
\end{equation*}
Of course, these are also equivalent to the standard scalar product and norm of $H^1(\R^2)$.
Then we define the limit functional
\begin{equation*}
  I_\infty:\: X \rightarrow \R, \qquad I_\infty(u) = \frac{1}{2} \|u\|_\star^2
    + \frac{1}{4} N_0(u) - \frac{b}{p}|u|_p^p,
\end{equation*}
and the associated Nehari manifold
\begin{equation*}
  \mathcal{N}_\infty := \left\{ u \in X \backslash \{0\}: \: I_\infty'(u) u =0 \right\}
  =\left\{ u \in X \backslash \{0\} : \: \|u\|_\star^2 + N_0(u) =b|u|_p^p \right\}.
\end{equation*}
Finally we define
\begin{equation*}
  m_\infty = \inf_{u \in \mathcal{N}_\infty}I_\infty(u).
\end{equation*}

Now we are ready to compare $m$ and $m_\infty$, which is crucial to show that
every minimizing sequence for $m$ is bounded in $X$.
\begin{lemma}\label{lem 3.4}
We have $m < m_\infty$.
\end{lemma}

\begin{proof}
In view of \cite[Theorem 1.1]{Cingolani-Weth-2016}, we know that $m_{\infty}>0$  can be attained at
a positive ground state solution $w \in X$ of \eqref{eq 3.4}.
By $(V_0)$, we conclude that
\begin{equation*}
  \int_{\R^2} V(x) w^2\:\textrm{d}x  < \int_{\R^2} V_\infty w^2\:\textrm{d}x,
\end{equation*}
which obviously implies that $\|w\|^2 < \|w\|_\star^2$. Then we have
\begin{equation*}
  I'(w)w = \|w\|^2 + N_0(w) - b|w|_p^p < \|w\|_\star^2 + N_0(w) - b|w|_p^p =I_\infty'(w)w=0.
\end{equation*}
Using Lemma \ref{lem 3.1}, we thus obtain that there exists a unique $t \in (0,1)$
such that $tw \in \mathcal{N}$, so that
\begin{align*}
  m&\leq I(tw)=\frac{1}{4} t^2\|w\|^2 + \left(\frac{b}{4}-\frac{b}{p}\right)t^p|w|^p_p\\
   &< \frac{1}{4}\|w\|^2 + \left(\frac{b}{4}-\frac{b}{p}\right)|w|^p_p\\
   &< \frac{1}{4}\|w\|_\star^2 + \left(\frac{b}{4}-\frac{b}{p}\right)|w|^p_p\\
   &= I_\infty(w) - \frac{1}{4}I_\infty'(w)w \\
   &= I_\infty(w) = m_\infty.
\end{align*}
This completes the proof.
\end{proof}

The following Proposition is the final step in the proof of Theorem \ref{th 1.1}.
\begin{proposition}\label{prop 3.5}
The level $m$ is achieved, and every minimizer of $m$ is a critical point
of $I$ which does not change sign on $\R^2$.
\end{proposition}

\begin{proof}
In the following, we divide our proof into three parts.

(i) We first prove that $m$ can be attained. Let $\{u_n\} \subset \mathcal{N}$
be a minimizing sequence for $I$, that is, $I(u_n) \rightarrow m$ as $n\rightarrow\infty$.
It follows from \eqref{eq 3.1} that $\{u_n\}$
is bounded in $H^1(\R^2)$. Then, passing to a subsequence if necessary,
there exists $u \in H^1(\R^2)$ such that
\begin{equation}\label{eq 3.5}
  u_n \rightharpoonup u \quad \text{in} \ H^{1}(\R^{2}), \quad
  u_n \rightarrow u \quad \text{in}\ L^s_{loc}(\R^{2})\ \text{for all}\ s \geq 1, \quad
  u_n(x)\rightarrow u(x) \quad \text{a.e. in}\ \R^2.
\end{equation}

We now claim that $u \neq 0$. Suppose by contradiction that $u=0$.
In view of \eqref{eq 3.5}, we see that $u_n \rightarrow 0$ in $L^2_{loc}(\R^{2})$.
Then, using the fact that $\{\|u_n\|\}$ is bounded, we can derive from $(V_0)$ that
\begin{equation*}
  \lim_{n \rightarrow \infty} \int_{\mathbb{R}^2}\left|V(x)-V_\infty\right|u^2_n\, \textrm{d}x=0.
\end{equation*}
Hence, by \eqref{eq 3.3} we have, after passing to a subsequence,
\begin{equation}\label{eq 3.6}
  \lim_{n \rightarrow \infty} \|u_n\|^2_{\star}
   = \lim_{n \rightarrow \infty} \|u_n\|^2:= \beta \in [\alpha, \infty).
\end{equation}
Using $(V_0)$ again, we yield that
\begin{equation*}
  I_\infty'(u_n)u_n = \|u_n\|_\star^2 + N_0(u_n) - b|u_n|_p^p
   \geq \|u_n\|^2 + N_0(u_n) - b|u_n|_p^p = 0.
\end{equation*}
From \cite[Lemma 2.5]{Cingolani-Weth-2016}, we thus deduce that for each $n \in \mathbb{N}$,
there exists $t_n \geq 1$ such that $t_n u_n \in \mathcal{N}_{\infty}$, that is,
\begin{equation}\label{eq 3.7}
  t_n^2 \|u_n\|^2_\star + t_n^4 N_0(u_n) = b t_n^p |u_n|_p^p.
\end{equation}
By the Sobolev inequality, we find that $\{|u_n|_p\}$ is bounded and,
up to a subsequence, we set
\begin{equation*}
  \lim_{n \rightarrow \infty} |u_n|_p^p :=\gamma \in [0, \infty).
\end{equation*}
Since $\{u_n\} \subset \mathcal{N}$, it then follows from \eqref{eq 3.6} and \eqref{eq 3.7} that
\begin{equation}\label{eq 3.8}
  \left(t_n^{-2} - 1\right) \left(\beta + o(1)\right)
  = b\left(t_n^{p-4}-1\right) \left(\gamma + o(1)\right)
  \qquad \text{for all} \ n \in \N,
\end{equation}
which implies that $t_n  \rightarrow 1$ as $n\rightarrow \infty$.
Consequently, we have
\begin{align*}
  m_\infty &\leq I_\infty(t_n u_n) =\frac{1}{4} t_n^2 \|u_n\|_\star^2
    + \left(\frac{b}{4}-\frac{b}{p}\right)t_n^p|u_n|_p^p\\
 &= \frac{1}{4}t_n^2 \|u_n\|^2 + \left(\frac{b}{4}-\frac{b}{p}\right)t_n^2 |u_n|_p^p
    +\left(\frac{b}{4}-\frac{b}{p}\right)\left(t_n^p - t_n^2\right)|u_n|_p^p + o(1)\\
 &= t_n^2 I(u_n) + o(1).
\end{align*}
Passing to the limit, we obtain $ m_\infty \leq m$, which contradicts Lemma \ref{lem 3.4}.
So $u \neq 0$, as claimed.

Since $\{u_n\} \subset \mathcal{N}$,  we see that
\begin{equation*}
  B_1\left(u_n^2, u_n^2\right) = N_1(u_n) =  N_2(u_n) + b|u_n|_p^p - \|u_n\|^2,
\end{equation*}
which implies that $\sup_{n \in \mathbb{N}} B_1\left(u_n^2, u_n^2\right) < \infty$ due to the boundedness
of $\{\|u_n\|\}$. Therefore, $|u_n|_\ast$ remains bounded in $n$ by Lemma \ref{lem 2.2}, and so $\{u_n\}$
is bounded in $X$.  Then, passing to a subsequence if necessary, we may assume that $u_{n} \rightharpoonup u$
in $X$, so that $u \in X$. It follows from Lemma \ref{lem 2.1}(i)
that $u_{n} \rightarrow u$ in $L^{s}(\mathbb{R}^{2})$ for $s \geq 2$. Using the weak lower semicontinuity
of the norm and Lemma \ref{lem 2.1}(iv), we thus derive that
\begin{equation}\label{eq 3.9}
  I(u) \leq \liminf_{n \rightarrow \infty} I(u_n) = m,
\end{equation}
and
\begin{equation}\label{eq 3.10}
  \|u\|^2 + N_1(u) \leq  N_2(u) + b|u|_p^p.
\end{equation}
If $\|u\|^2 + N_1(u) = N_2(u) + b|u|_p^p$, then $u \in \mathcal{N}$ and \eqref{eq 3.9} immediately shows that $m$
is achieved at $u$. Since \eqref{eq 3.10} holds, we only have to treat the case where
\begin{equation}\label{eq 3.11}
  \|u\|^2 + N_1(u) <  N_2(u) + b|u|_p^p.
\end{equation}
We now prove that if \eqref{eq 3.11} occurs, it leads to a contradiction. Indeed, we know from Lemma \ref{lem 3.1}
and \eqref{eq 3.11} that there exists a unique $t \in (0,1)$ such that $tu \in \mathcal{N}$, and hence
\begin{align*}
  m&\leq I(tu)=\frac{1}{4} t^2\|u\|^2 + \left(\frac{b}{4}-\frac{b}{p}\right)t^p|u|^p_p\\
   &< \frac{1}{4}\|u\|^2 + \left(\frac{b}{4}-\frac{b}{p}\right)|u|^p_p\\
   &\leq \liminf_{n \rightarrow \infty}\left[\frac{1}{4}\|u_n\|^2 + \left(\frac{b}{4}-\frac{b}{p}\right)|u_n|_p^p\right]\\
   &=\liminf_{n\rightarrow \infty} I(u_n) = m.
\end{align*}
This is impossible, and part (i) is thus proved.

(ii) We next prove that every minimizer of $m$ is a critical point of $I$.
Let $u \in \mathcal N$ be an arbitrary minimizer for $I$ on
$\mathcal N$. We show that $I'(u) v =0$ for all $v \in X$, so that $u$ is a critical point of $I$.

For every $v \in X$, there exists $\varepsilon >0$ such that $u+s\upsilon \neq 0$ for all $s\in(-\varepsilon, \varepsilon)$.
We now consider the function $ \varphi: \: (-\varepsilon, \varepsilon) \times (0, \infty) \rightarrow \R$ defined by
\begin{equation*}
  \varphi(s, t)= t^2\|u + s\upsilon\|^2 + t^4 N_0(u + s\upsilon) - bt^p|u + s\upsilon|_p^p.
\end{equation*}
Since $u \in \mathcal{N}$, we have $\varphi(0,1)=0$. Moreover, $\varphi$ is a $C^1$-function and
\begin{equation*}
  \frac{\partial \varphi}{\partial t}(0,1)=2\|u\|^2 + 4N_0(u)-pb|u|_p^p = -2\|u\|^2+(4-p)b|u|^p_p<0.
\end{equation*}
Therefore, by the implicit function theorem, for $\varepsilon$ small enough we can determine a $C^1$-function
$t: (-\varepsilon, \varepsilon) \rightarrow \mathbb{R}$ such that $t(0)=1$ and
\begin{equation*}
  \varphi\left(s,t(s)\right)=0 \qquad \text{for all} \ s \in (-\varepsilon, \varepsilon).
\end{equation*}
This also shows that $t(s) > 0$, at least for $\varepsilon$ very small, and so $t(s)(u+sv) \in \mathcal{N}$.
Then, we define $\gamma:\: (-\varepsilon,  \varepsilon) \rightarrow \R$ as
\begin{equation*}
  \gamma(s)=I\left(t(s)(u+s\upsilon)\right).
\end{equation*}
Clearly, the function $\gamma$ is differentiable and has a minimum point at $s=0$, and thus
\begin{equation*}
  0 = \gamma'(0) = I'\left(t(0)u\right)\left(t'(0)u+t(0)v\right)
    = t'(0)I'(u)u + I'(u)v =I'(u)v.
\end{equation*}
Since $v \in X$ is arbitrary, we conclude that $I'(u)=0$. Hence, part (ii) follows.

(iii) We finally prove that every minimizer of $m$ does not change sign in $\R^2$.
If $u \in \mathcal N$ is a minimizer of $I|_{\mathcal N}$,
then $|u|$ is also a minimizer of $I|_{\mathcal N}$ due to the fact that
$|u| \in \mathcal{N}$ and $I(u)=I(|u|)$.
So, $|u|$ is a critical point of $I$ by the considerations above.
Using the standard elliptic regularity theory, we find that
$|u| \in C_{loc}^{1,\alpha}(\mathbb{R}^{2})$ for every $\alpha \in (0, 1)$
and $-\Delta |u| + q(x) |u| = 0$ in $\mathbb{R}^{2}$
with some function $q(x) \in L_{loc}^{\infty}(\mathbb{R}^{2})$.
Therefore, the strong maximum principle and the fact that $u \neq 0$ imply that
$|u|>0$ in $\mathbb{R}^{2}$, which shows that $u$ does not change sign in $\R^2$.
The proof is thus finished.
\end{proof}

The {\em proof of Theorem \ref{th 1.1}} is now completed by combining Proposition \ref{prop 3.5}
and Corollary \ref{coro 3.2}.

\section{Proof of Theorem \ref{th 1.2}} \label{sec 4}

\indent

In this section, we are devoted to the proof of Theorem \ref{th 1.2} on the existence of
ground state solutions for \eqref{eq 1.6} in the case where $3 \leq p < 4$.
To this aim, we will first prove the existence of mountain pass solutions for \eqref{eq 1.6}.
Within this step, we shall employ the following general minimax principle from \cite{li-wang:11}.
It is a somewhat stronger version of \cite[Theorem 2.8]{Willem-1996},
which leads to Cerami sequences instead of Palais-Smale sequences.

\begin{proposition}[\hspace{-0.05ex}{\cite[Proposition 2.8]{li-wang:11}}]\label{prop 4.1}
Let $X$ be a Banach space. Let $M_{0}$ be a closed subspace of the metric
space $M$ and $\Gamma_{0} \subset C(M_{0}, X)$. Define
\vskip -0.2 true cm
\begin{equation*}
  \Gamma = \left\{\gamma \in C(M, X):\: \gamma|_{M_{0}} \in \Gamma_{0} \right\}.
\end{equation*}
If $\varphi \in C^{1}(X, \R)$ satisfies
\begin{equation*}
  \infty > c := \inf_{\gamma \in \Gamma} \sup_{u \in M} \varphi
    \left(\gamma (u)\right) > a:= \sup_{\gamma_{0}\in \Gamma_{0}}
    \sup_{u \in M_{0}} \varphi \left(\gamma_{0}(u)\right),
\end{equation*}
then, for every $\varepsilon \in \left(0, \frac{c-a}{2}\right)$, $\delta>0$ and
$\gamma \in \Gamma $ with $\sup_{u \in M}\varphi\left(\gamma(u)\right) \leq c + \varepsilon$,
 there exists $u \in  X$ such that
\begin{itemize}
  \item [$(a)$] $c - 2\varepsilon \leq \varphi(u) \leq c + 2\varepsilon$,

  \item [$(b)$] ${\rm dist} \left(u, \gamma(M)\right) \leq 2\delta$,

  \item [$(c)$] $\left(1+\|u\|_{X}\right)\|\varphi'(u)\|_{X'} \leq \frac{8\varepsilon}{\delta}$.
\end{itemize}
\end{proposition}

We now define the mountain pass level of $I$ by
\begin{equation*}
  c = \inf_{\gamma \in \Gamma} \max_{t \in [0,1]}I\left(\gamma(t)\right),
\end{equation*}
where
\begin{equation*}
  \Gamma := \left\{\gamma\in C\left([0,1], X\right): \: \gamma(0)=0, \, I(\gamma(1))<0\right\}.
\end{equation*}
By Lemmas \ref{lem 2.5} and \ref{lem 2.6}, we find that the functional $I$ possesses a mountain pass geometry,
and moreover,
\begin{equation}\label{eq 4.1}
  0 < m_{\rho} \leq c < \infty.
\end{equation}
In the following, we always assume that $b \geq 0$, $p >2$, and that $(V_0)$ and $(V_1)$ hold.
Similarly as in \cite[Lemma 3.2]{Du-Weth-2017}, we shall make use of Proposition \ref{prop 4.1}
to produce a Cerami sequence $\{u_{n}\}\subset X$  at the energy level $c$ with a key additional property,
from which we can easily conclude the boundedness of $\{\|u_n\|\}$.
For Palais-Smale sequences in related variational settings, this idea traces back to \cite{Jeanjean-1997}
and has also been used successfully in \cite{Hirata-2010,Moroz-VanSchaftingen-2015}.

\begin{lemma}\label{lem 4.2}
There exists a sequence $\{u_{n}\}$ in $X$ such that,
as $n\rightarrow\infty$,
\begin{equation}\label{eq 4.2}
  I(u_{n})\rightarrow c,
  \quad \|I'(u_{n})\|_{X'}\left(1+\|u_{n}\|_{X}\right) \rightarrow 0
  \quad \text{and}\quad J(u_{n}) \rightarrow 0,
\end{equation}
where $J: X \to \R$ is defined in \eqref{eq 1.9}.
\end{lemma}

\begin{proof}
Borrowing the ideas from \cite{Jeanjean-1997} (see also \cite{Du-Weth-2017}), we consider the Banach space
\begin{equation*}
  \tilde X := \mathbb{R} \times X
\end{equation*}
endowed with the product norm $\|(s, v)\|_{\tilde X}:=\left(|s|^{2}+\|v\|^2_{X} \right)^{1/2}$ for $(s, v) \in \tilde{X}$.
Moreover, we define the continuous map
\begin{equation*}
  h: \: \tilde X \rightarrow X, \qquad h(s, v)(\cdot) =e^{2s} v \left(e^{s}\cdot\right)
  \quad \text{for} \ (s, v) \in \tilde{X}.
\end{equation*}
We also consider the functional
\begin{equation*}
  \varphi := I \circ h :\: \tilde X \to \R,
\end{equation*}
a simple calculation then yields that
\begin{align}\label{eq 4.3}
  \varphi(s, v)&=\frac{e^{4s}}{2}\int_{\mathbb{R}^{2}} |\nabla v|^{2} \,\textrm{d}x
    + \frac{e^{2s}}{2}\int_{\mathbb{R}^{2}}V(e^{-s}x)v^{2}\,\textrm{d}x
    + \frac{e^{4 s}}{8 \pi}\int_{\mathbb{R}^{2}}\int_{\mathbb{R}^{2}}
    \log\left(|x-y|\right)v^{2}(x)v^{2}(y) \,\textrm{d}x\textrm{d}y \nonumber\\
  &\quad-\frac{s e^{4s}}{8 \pi}\left(\int_{\mathbb{R}^{2}}|v|^{2}\,\textrm{d}x\right)^{2}
    -\frac{b e^{2s(p-1)}}{p}\int_{\mathbb{R}^{2}}|v|^{p}\,\textrm{d}x
    \qquad \text{for} \ (s, v) \in \tilde{X}.
\end{align}
It is easy to verify that $\varphi$ is of class $C^1$ on $\tilde{X}$ with
\begin{equation}\label{eq 4.4}
  \partial_{s} \varphi(s, v) = J\left(h(s, v)\right)
   \qquad \text{and} \qquad
  \partial_{v}  \varphi(s, v) w = I'\left(h(s, v)\right)h(s, w)
\end{equation}
for $s \in \R$ and $v, w \in X$. Next, we define the minimax level of $\varphi$ by
\begin{equation*}
  \tilde{c} = \inf_{\tilde{\gamma }\in \tilde{\Gamma}}\max_{t \in [0,1]}
    \varphi(\tilde{\gamma}(t)),
\end{equation*}
where
\begin{equation*}
  \tilde{\Gamma}:=\left\{\tilde{\gamma} \in C\bigl([0, 1],  \tilde X\bigr) : \:
  \tilde{\gamma} (0)=(0, 0), \ \varphi\left(\tilde{\gamma}(1)\right)<0 \right\}.
\end{equation*}
Since $\Gamma =\{ h \circ \tilde{\gamma} :\: \tilde{\gamma}\in \tilde{\Gamma}\}$,
the minimax levels of $I$ and $\varphi$ coincide, i.e., $c = \tilde{c}$.
In view of \eqref{eq 4.1}, we can use Proposition \ref{prop 4.1} to the functional $\varphi$, $M= [0,1]$,
$M_0= \{0,1\}$ and $\tilde X$, $\tilde \Gamma$ in place of $X$, $\Gamma$.
More precisely, by the definition of $c$, for every fixed $n \in \mathbb N$ there exists
$\gamma_{n}\in\Gamma$ such that
\begin{equation*}
  \max_{t\in [0,1]} I\left(\gamma_{n}(t)\right) \leq c + \frac{1}{n^{2}}.
\end{equation*}
Then we define $\tilde \gamma_n \in \tilde \Gamma$ by $\tilde{\gamma}_{n}(t)=\left(0, \gamma_{n}(t)\right)$,
so that
\begin{equation*}
  \max_{t\in[0,1]} \varphi \left(\tilde{\gamma}_{n}(t)\right)
  = \max_{t\in[0,1]} I\left(\gamma_{n}(t)\right)\leq c + \frac{1}{n^{2}}.
\end{equation*}
As a direct application of Proposition \ref{prop 4.1} with $\tilde \gamma_n$ in place of $\gamma$ and $\varepsilon=\frac{1}{n^{2}}$,
$\delta=\frac{1}{n}$, we therefore obtain a sequence $(s_{n}, v_{n})\in \tilde X$ such that, as $n \to \infty$,
\begin{align}
  &\varphi (s_{n}, v_{n})\rightarrow c, \label{eq 4.5} \\
  &\left\|\varphi'(s_{n},  v_{n})\right\|_{\tilde X'}
     \left(1 + \left\|(s_{n},  v_{n})\right\|_{\tilde X}\right)
     \rightarrow 0, \label{eq 4.6}\\
  &{\rm dist}\bigl((s_{n},  v_{n}),  \{0\} \times \gamma_{n}([0,1])\bigr) \to 0, \label{eq 4.7}
\end{align}
and \eqref{eq 4.7} readily implies that
\begin{equation}\label{eq 4.8}
  s_{n} \rightarrow 0.
\end{equation}
Observe from \eqref{eq 4.4} that
\begin{equation}\label{eq 4.9}
  \varphi'(s_{n},  v_{n}) (k, w)
   = I'\left(h(s_{n}, v_{n})\right) h(s_{n}, w) + J\left( h(s_{n},  v_{n})\right)k
   \qquad \text{for} \ (k, w) \in \tilde X,
\end{equation}
by taking $k = 1$ and $w = 0$ in \eqref{eq 4.9} we deduce from \eqref{eq 4.6} that
\begin{equation}\label{eq 4.10}
  J\left(h (s_{n}, v_{n}) \right)\rightarrow 0
    \qquad \text{as}\ n \rightarrow \infty.
\end{equation}
According to \eqref{eq 4.5} and \eqref{eq 4.10}, for $u_{n}:= h (s_{n}, v_{n})$ we have
\begin{equation*}
   I(u_{n})\rightarrow c \quad \text{and} \quad
   J(u_{n})\rightarrow 0 \qquad \text{as}\ n \rightarrow \infty.
\end{equation*}
Finally, for given $v \in X$ we define $w_n = e^{-2 s_{n}}v(e^{-s_{n}} \cdot) \in X$,
and then conclude from \eqref{eq 4.6}, \eqref{eq 4.8} and \eqref{eq 4.9} with $k=0$ that
\begin{equation*}
  \left(1+\|u_{n}\|_{X}\right)\left|I'(u_{n})v \right| = \left(1+\|u_{n}\|_{X}\right)\left|I'(u_{n})h(s_n, w_n) \right|
  = o(1)\|w_n \|_{X} \qquad \text{as $n \to \infty$,}
\end{equation*}
whereas by \eqref{eq 4.8} we obtain
\begin{align*}
  \|w_n \|_{X}^2
 &= e^{-4s_n} \int_{\R^2} |\nabla v|^{2}\: \textrm{d}x+ e^{-2s_n}\int_{\R^2}\bigl[V(e^{s_n} x)+ \log\left(1 + e^{s_n}|x|\right)\bigr]v^{2}\:\textrm{d}x\\
 &\leq \left(\frac{V_\infty}{V_0} + o(1)\right)\|v\|_{X}^2 \qquad \text{as} \ n \to \infty
\end{align*}
with $o(1) \to 0$ uniformly in $v \in X$. Combining the latter two estimates gives
\begin{equation*}
  \left(1+\|u_{n}\|_{X}\right)\|I'(u_{n})\|_{X'} \to 0 \qquad \text{as $n \to \infty$.}
\end{equation*}
This completes the proof.
\end{proof}

In the following key lemma, we shall show, in particular, that any sequence $\{u_{n}\}$
satisfying \eqref{eq 4.2} is bounded in $H^{1}(\mathbb{R}^{2})$.

\begin{lemma}\label{lem 4.3}
Let $d \in \R$, and let $\{u_{n}\} \subset X$ be a sequence such that
\begin{equation}\label{eq 4.11}
  I(u_{n}) \rightarrow d \qquad \text{and} \qquad
  J(u_{n}) \rightarrow 0 \qquad \text{as} \  n \rightarrow \infty.
\end{equation}
Then $\{u_{n}\}$ is bounded in $H^{1}(\mathbb{R}^{2})$.
\end{lemma}

\begin{proof}
By $(V_0)$ and $(V_1)$, we can derive from \eqref{eq 4.11} that
\begin{equation}\label{eq 4.12}
  d + o(1) = I(u_n) - \frac{1}{4} J(u_n)
   \geq \frac{2V_0 - \eta}{8} |u_n|_2^2 + \frac{1}{32 \pi} |u_n|_2^4 + \frac{b(p-3)}{2p}|u_n|_p^p.
\end{equation}
Then we may distinguish the following two cases:

\emph{Case 1: $b>0$ and $p>3$}. In this case, \eqref{eq 4.12} obviously implies that $\{u_{n}\}$ is bounded in $L^{2}(\R^2)$
and in $L^{p}(\mathbb{R}^{2})$. It then follows from \eqref{eq 2.2} and the H\"{o}lder inequality that
\begin{equation*}
  N_2(u_n) \leq C_0 |u_n|^4_{\frac{8}{3}} \leq C_0 |u_n|^{4(1-\theta_0)}_2 |u_n|^{4\theta_0}_p \leq C_1,
\end{equation*}
where $\theta_0 = \frac{p}{4(p-2)}$. Using \eqref{eq 4.11} again, we therefore deduce that
\begin{equation*}
  2\|u_n\|^2 + N_1(u_n) =  4I(u_n) + N_2(u_n) + \frac{4b}{p}|u_n|^p_p
  \leq  4d + C_1 + \frac{4b}{p} |u_n|_p^p + o(1) \leq C_2 + o(1).
\end{equation*}
This implies that $\{u_n\}$ is bounded in $H^1(\R^2)$.

\emph{Case 2: $b=0$ or $2< p \leq 3$}. We first claim that
\begin{equation}\label{eq 4.13}
  |\nabla u_n |_2 \leq C_3  \qquad \text{for $n \in \mathbb{N}$}.
\end{equation}
On the contrary, suppose that \eqref{eq 4.13} does not occur.
We then have, up to a subsequence,
\begin{equation*}
  |\nabla u_n|_2 \rightarrow \infty  \qquad \text{as $n \rightarrow \infty$}.
\end{equation*}
Let $t_n:=|\nabla u_n|_2^{-1/2}$ for $n \in \mathbb{N}$, so that $t_n \rightarrow 0$
as $n \rightarrow \infty$. For $n \in \mathbb{N}$, we define the rescaled function
$v_n \in X$ by $v_n(x):= t_n^2 u_n(t_n x)$ for $x \in \R^2$, and
we can easily see that
\begin{equation}\label{eq 4.14}
  |\nabla v_n|_2 =1 \qquad \text{and} \qquad |v_n|_q ^q = t_n^{2q-2} |u_n|_q ^q
  \quad \text{with }  1 \leq q < \infty.
\end{equation}
Therefore, by the Gagliardo-Nirenberg inequality we obtain
\begin{equation}\label{eq 4.15}
  |v_n|^p_p \leq C_4 |v_n|^2_2 |\nabla v_n|^{p-2}_2 = C_4 |v_n|^2_2 \qquad \text{for $n \in \mathbb{N}$}.
\end{equation}
Multiplying \eqref{eq 4.12} by $t^4_n$, we conclude from \eqref{eq 4.14} and \eqref{eq 4.15} that
\begin{align*}
  d t_n^4 + o(t_n^4) &\geq \frac{2V_0 - \eta}{8} t_n^4 |u_n|_2^2
      + \frac{1}{32 \pi} t_n^4 |u_n|^4_2 - \frac{b(3-p)}{2p} t_n^4 |u_n|_p^p \\
   &\geq \frac{2V_0 - \eta}{8} t_n^2 |v_n|_2^2  + \frac{1}{32 \pi} |v_n|^4_2 -
    \frac{b(3-p)}{2p} C_4 t_n^{6-2p} |v_n|^2_2.
\end{align*}
Consequently,
\begin{equation}\label{eq 4.16}
  |v_{n}|_{2}=
  \begin{cases}
    o\bigl(t_{n}^{1/2}\bigr)      \qquad &\text{if $b=0$ or $p=3$},\\[1.5mm]
    o\bigl(t_{n}^{(3-p)/2}\bigr)  &\text{if $b>0$ and $2<p<3$}.
  \end{cases}
\end{equation}
Moreover, by assumption we also have
\begin{align*}
  o(1) = t^4_n J(u_n) =&  t_n^4 \biggl( 2|\nabla u_n |_2 ^2
   + \int_{\R^2} \Bigl[ V(x)- \frac{1}{2}(\nabla V(x), x)\Bigr]u_n^2 \,\textrm{d}x
    +  N_0(u_n) \\
   &- \frac{1}{8 \pi} |u_n|_2^4 - \frac{2b(p-1)}{p}|u_n|_p^p \biggr).
\end{align*}
Combining this with \eqref{eq 4.14}$-$\eqref{eq 4.16} and the fact that
\begin{equation*}
  N_{0}(u_{n})= \frac{t_n^4}{2\pi} \int_{\R^2} \int_{\R^2}
   \log \left(|t_n x-t_n y|\right)u_n^{2}(t_n x)u_n^{2}(t_n y)\:\textrm{d}x\textrm{d}y
  = t_n^{-4}\left( N_0(v_n) + \frac{\log t_n}{2\pi}|v_n|_2^4 \right),
\end{equation*}
we obtain
\begin{equation}\label{eq 4.17}
  o(1) = 2 + N_0(v_n) + \frac{\log t_n}{2 \pi} |v_n|_2 ^4 + o(1)
    = 2 + N_0(v_n) + o(1).
\end{equation}
From \eqref{eq 2.2}, \eqref{eq 4.16}, \eqref{eq 4.17} and the Gagliardo-Nirenberg inequality,
we thus deduce that
\begin{equation*}
  2 \leq 2 + N_1(v_n) = N_2(v_n)+ o(1) \leq C_0 |v_n|^4_{\frac{8}{3}}+ o(1) \leq C_5 |v_n|^3_2+ o(1) = o(1),
\end{equation*}
which leads to a contradiction, and hence \eqref{eq 4.13} holds. It then follows from
the Gagliardo-Nirenberg inequality and \eqref{eq 4.12} that
\begin{equation*}
  d + o(1) \geq \frac{2V_0 - \eta}{8} |u_n|_2^2 + \frac{1}{32 \pi}  |u_n|^4_2 - \frac{b(3-p)}{2p}|u_n|^p_p
   \geq \frac{1}{32 \pi} |u_n|^4_2  -  C_6 |u_n|_2^2.
\end{equation*}
Therefore, $\{u_n\}$ is bounded in $L^2(\R^2)$. This, together with \eqref{eq 4.13},
gives that $\{u_{n}\}$ is bounded in $H^{1}(\R^2)$, as claimed.
\end{proof}

Next, we investigate the compactness property for the Cerami sequence satisfying \eqref{eq 4.2}.
For this we need to define the the mountain pass level of $I_\infty$ by
\begin{equation*}
  c_\infty = \inf_{\gamma \in \Gamma_\infty} \max_{t \in [0,1]} I_\infty \bigl(\gamma(t)\bigr),
\end{equation*}
where
\begin{equation*}
  \Gamma_\infty := \left\{\gamma\in C\left([0,1],  X\right) : \: \gamma(0)=0, \ I_\infty(\gamma(1))<0\right\}.
\end{equation*}

\begin{proposition}\label{prop 4.4}
Suppose that $b \geq 0$, $p \geq 3$, and that $(V_0)$ and $(V_1)$ hold.
Let $d \in (-\infty, c_\infty)$, and let $\{u_{n}\} \subset X$ be a sequence such that
\begin{equation}\label{eq 4.18}
  I(u_{n}) \rightarrow d, \qquad
  \|I'(u_{n})\|_{X'}\left(1+\|u_{n}\|_{X}\right)\rightarrow 0 \qquad \text{and}  \qquad
  J(u_{n}) \rightarrow 0 \qquad \text{as} \  n \rightarrow \infty.
\end{equation}
Then, up to a subsequence, one of the following holds:
\begin{itemize}
  \item[\rm(I)] $\|u_n\| \to 0$ and $I(u_n) \to 0$ as $n \to \infty$.

  \item[\rm(II)] There exists $u \in X$ such that $u_n \rightarrow u$ in $X$ as $n \to \infty$.
\end{itemize}
\end{proposition}

\begin{proof}
By Lemma \ref{lem 4.3}, we know that $\{u_n\}$ is bounded in $H^1(\R^2)$.
Suppose now that (I) does not hold for any subsequence of $\{u_n\}$.
To finish the proof, it suffices to show that, up to a subsequence, (II) must occur.
For this we first claim that
\begin{equation}\label{eq 4.19}
  \liminf_{n \rightarrow \infty}\sup_{y \in \mathbb{R}^{2}}
    \int_{B_{2}(y)} |u_{n}|^{2}\,\textrm{d}x>0.
\end{equation}
Suppose by contradiction that \eqref{eq 4.19} is false. Then, after passing to a subsequence,
Lions' vanishing lemma (see \cite[Lemma I.1]{Lions-1984} or \cite[Lemma 1.21]{Willem-1996})
says that $u_{n} \rightarrow 0$ in $L^{s}(\R^2)$ for all $s>2$.
From \eqref{eq 2.2} and \eqref{eq 4.18}, we therefore deduce that
\begin{equation*}
  \|u_n\|^{2} + N_{1}(u_{n}) = I'(u_{n})u_{n} + N_{2}(u_{n}) + b|u_{n}|_{p}^{p}
    \rightarrow 0 \qquad \text{as}\ n \rightarrow \infty.
\end{equation*}
Consequently, we have $\|u_{n}\| \rightarrow 0$ and $N_{1}(u_{n}) \rightarrow 0$,
so that
\begin{equation*}
  I(u_{n})=\frac{1}{2}\|u_{n}\|^{2}+\frac{1}{4}\left(N_{1}(u_{n})
    - N_{2}(u_{n})\right) - \frac{b}{p}|u_{n}|_{p}^{p} \rightarrow 0
  \qquad  \text{as}\ n \rightarrow \infty.
\end{equation*}
This contradicts our assumption that (I) does not hold for any subsequence of $\{u_n\}$,
and so the claim follows. Going if necessary to a subsequence,
there exists a sequence $\{y_{n}\} \subset \mathbb{R}^{2}$ such that,
the sequence of the functions
\begin{equation*}
  \tilde{u}_{n}:= u_{n} (\cdot + y_n) \in X \quad \text{with} \ n \in \N,
\end{equation*}
converges weakly in $H^{1}(\R^{2})$ to some function $\tilde{u} \in H^{1}(\R^{2})\setminus \{0\}$,
so that $\tilde{u}_{n}(x) \rightarrow \tilde{u}(x)$ a.e. in $\R^{2}$.
Moreover, invoking \eqref{eq 4.18} again, we conclude that
\begin{equation*}
  B_{1}\left(\tilde{u}_{n}^{2}, \tilde{u}_{n}^{2}\right) = N_{1}(\tilde{u}_{n})
    =N_{1}(u_{n})=o(1)+N_{2}(u_{n}) + |u_{n}|_{p}^{p}-\|u_{n}\|^{2},
\end{equation*}
and the RHS of this equality remains bounded in $n$. By Lemma \ref{lem 2.2},
$\{|\tilde{u}_{n}|_{\ast}\}$ is bounded, and thus $\{\tilde{u}_{n}\}$
is bounded in $X$. Then, passing to a subsequence again if necessary, we may assume that
$\tilde{u}_{n} \rightharpoonup \tilde{u}$ in $X$, so that $\tilde{u} \in X$.
It now follows from Lemma \ref{lem 2.1}(i) that $\tilde{u}_{n} \rightarrow \tilde{u}$
in $L^{s}(\R^2)$ for $s \geq 2$. Therefore, we have for every $n \in \N$,
\begin{equation}\label{eq 4.20}
  \left|I'(\tilde{u}_{n})(\tilde{u}_{n}-\tilde{u})\right|\
    \leq \|I'(u_{n})\|_{X'} \left(\|u_{n}\|_{X}+\|\tilde{u}(\cdot -y_n)\|_{X}\right) + o(1).
\end{equation}
Note that for any $v \in X$, by \eqref{eq 4.19} a crude estimate gives
\begin{equation}\label{eq 4.21}
  \|v(\cdot -y_n)\|_X \leq C_1 \|u_n\|_X\qquad \text{for $n$ large enough}.
\end{equation}
Combining this with \eqref{eq 4.18} and \eqref{eq 4.20}, we infer that
\begin{equation*}
  I'(\tilde{u}_{n})(\tilde{u}_{n}-\tilde{u}) \rightarrow 0
  \qquad \text{as} \ n \rightarrow \infty.
\end{equation*}
This implies that
\begin{align*}
  o(1) &= I'(\tilde{u}_{n})\left(\tilde{u}_{n}-\tilde{u}\right)\\
   &=\|\tilde{u}_{n}\|^{2}-\|\tilde{u}\|^{2}+\frac{1}{4}N'_{0}(\tilde{u}_{n})
    (\tilde{u}_{n}-\tilde{u}) - b\int_{\mathbb{R}^{2}}|\tilde{u}_{n}|^{p-2}
    \tilde{u}_{n}(\tilde{u}_{n} - \tilde{u})\,\textrm{d}x + o(1)\\
   &=\|\tilde{u}_{n}\|^{2}-\|\tilde{u}\|^{2}+\frac{1}{4}\left [N'_{1}(\tilde{u}_{n})
    (\tilde{u}_{n}-\tilde{u})- N'_{2}(\tilde{u}_{n})(\tilde{u}_{n}-\tilde{u})\right] + o(1),
\end{align*}
where
\begin{equation*}
  \left|\frac{1}{4}N'_{2}(\tilde{u}_{n})(\tilde{u}_{n}-\tilde{u})\right|
   = \left|B_{2}\left(\tilde{u}_{n}^{2}, \tilde{u}_{n}(\tilde{u}_{n}-\tilde{u})\right)\right|
    \leq C_0 |\tilde{u}_{n}|_{\frac{8}{3}}^{3}\left|\tilde{u}_{n}-\tilde{u}\right|_{\frac{8}{3}}
    \rightarrow 0 \quad \text{as}\ n\rightarrow \infty
\end{equation*}
by \eqref{eq 2.1}, and
\begin{equation*}
  \frac{1}{4}N'_{1}(\tilde{u}_{n})(\tilde{u}_{n}-\tilde{u})
   = B_{1}\left(\tilde{u}_{n}^{2}, \tilde{u}_{n}(\tilde{u}_{n}-\tilde{u})\right)
   = B_{1}\left(\tilde{u}_{n}^{2}, (\tilde{u}_{n}-\tilde{u})^{2}\right)
    + B_{1}\left(\tilde{u}_{n}^{2}, \tilde{u}(\tilde{u}_{n}-\tilde{u})\right)
\end{equation*}
with
\begin{equation*}
  B_{1}\left(\tilde{u}_{n}^{2}, \tilde{u}(\tilde{u}_{n}-\tilde{u})\right) \rightarrow 0
    \quad \text{as}\ n\rightarrow\infty
\end{equation*}
in view of Lemma \ref{lem 2.3}. Combining these estimates, we have
\begin{equation*}
  o(1) = \|\tilde{u}_{n}\|^{2} - \|\tilde{u}\|^{2}
   + B_{1}\left(\tilde{u}_{n}^{2}, \,(\tilde{u}_{n}-\tilde{u})^{2}\right) + o(1)
  \geq \|\tilde{u}_{n}\|^{2}-\|\tilde{u}\|^{2} + o(1),
\end{equation*}
which means that $\|\tilde{u}_{n}\| \rightarrow \|\tilde{u}\|$ and
$B_{1}\left(\tilde{u}_{n}^{2}, (\tilde{u}_{n}-\tilde{u})^{2}\right) \rightarrow 0$
as $n \rightarrow \infty$. Hence,  $\|\tilde{u}_{n}-\tilde{u}\|\rightarrow0$
as $n \rightarrow \infty$. Moreover, by Lemma \ref{lem 2.2} we obtain
$|\tilde{u}_{n} - \tilde{u}|_{\ast} \rightarrow 0$, and thus
$\|\tilde{u}_{n} - \tilde{u} \|_{X} \rightarrow 0$ as $n \rightarrow \infty$.

Next, we claim that $\{y_n\}$ is bounded in $\R^2$. Indeed, if this is false,
then there exists a subsequence of $\{y_n\}$, still denoted by $\{y_n\}$,
such that $|y_n|\rightarrow \infty$ as $n \rightarrow \infty$.
By $(V_0)$ and \eqref{eq 4.18}, we derive that
\begin{equation}\label{eq 4.22}
  I_\infty(\tilde{u})= \lim_{n \rightarrow \infty} I_\infty (\tilde{u}_{n})
   = \lim_{n \rightarrow \infty} I (u_{n}) = d.
\end{equation}
Moreover, by \eqref{eq 4.21} we also have, for every $v \in X$,
\begin{align*}
  |I_\infty'(\tilde{u})v|& = \lim_{n \rightarrow \infty} |I_\infty'(\tilde{u}_{n})v|
    = \lim_{n \rightarrow \infty}\left|I'(u_{n}) v(\cdot - y_n)\right| \\
    & \leq  \lim_{n \rightarrow \infty}\|I'(u_{n})\|_{X'}
      \| v(\cdot - y_n)\|_{X} \leq C_{1}\lim_{n \rightarrow \infty}
      \|I'(u_{n})\|_{X'}\|u_{n}\|_{X} =0,
\end{align*}
which readily implies that $I_\infty'(\tilde{u})=0$. This, together with
\eqref{eq 4.22} and \cite[Theorem 1.2]{Du-Weth-2017}, gives
\begin{equation*}
  d = I_\infty(\tilde{u}) \geq c_\infty,
\end{equation*}
contradicting the assumption that $d < c_\infty$. As a consequence,
$\{y_n\}$ is bounded in $\R^2$, as claimed.

Finally, we show that (II) holds. Since $\{y_n\}$ is bounded in $\R^2$,
there exists $y_0 \in \R^2$ such that $y_n \rightarrow y_0$ as $n\rightarrow \infty$,
up to a subsequence. Set
\begin{equation*}
  u(x) := \tilde{u} (x - y_0) \quad \text{for} \ x \in \R^2,
\end{equation*}
then $u \in X$. Observe that, $\tilde{u}_n \rightarrow \tilde{u}$ in $X$ as $n \rightarrow \infty$,
we find that
\begin{equation*}
  u_n \rightarrow u \quad \text{in $X$ as $n \rightarrow\infty$}.
\end{equation*}
The proof is thus finished.
\end{proof}

The last step consists in showing that $c < c_\infty$.

\begin{lemma}\label{lem 4.5}
Suppose that $b \geq 0$, $p \geq 3$, and that $(V_0)$ and $(V_1)$ hold.
Then we have $c < c_\infty$.
\end{lemma}

\begin{proof}
By \cite[Theorems 1.1 and 1.2]{Du-Weth-2017}, we have that $c_\infty >0$ can be achieved
at a positive ground state solution $w \in X$ of \eqref{eq 3.4}.
Moreover, from \cite[Lemma 4.2]{Du-Weth-2017} we see that
\begin{equation}\label{eq 4.23}
  I_\infty (w) = \max_{t>0} I_\infty (w_t).
\end{equation}
Recalling the definition of $c$, we can deduce from $(V_0)$, \eqref{eq 4.23}
and Lemma \ref{lem 2.6} that
\begin{equation*}
  c \leq \max_{t > 0}I(w_t) < \max_{t>0} I_\infty (w_t) = I_\infty(w) = c_\infty.
\end{equation*}
Thus, the claim follows.
\end{proof}

\begin{proof}[Proof of Theorem \ref{th 1.2}]
By Lemma \ref{lem 4.2},  Proposition \ref{prop 4.4} and Lemma \ref{lem 4.5},
there exists a critical point $u \in X \setminus \{0\}$ of $I$ with $I(u) = c$.
In particular, the set
\begin{equation*}
  \mathcal{K}:=\{u \in X \backslash \{0\}: \: I'(u)=0 \}
\end{equation*}
is nonempty. Let $\{u_n\} \subset \mathcal{K}$ be a sequence such that
\begin{equation*}
  I(u_n) \to c_{g} := \inf_{u \in \mathcal{K}} I(u) \ \in (-\infty, c].
\end{equation*}
From the definition of $\mathcal{K}$ and Lemma \ref{lem 2.4}, we can easily see that
the sequence $\{u_n\}$ satisfies \eqref{eq 4.18}. Moreover, by \eqref{eq 2.4} we obtain
\begin{equation*}
  \liminf_{n \to \infty} \|u_n\| \geq \rho > 0.
\end{equation*}
It therefore follows from Proposition \ref{prop 4.4} and Lemma \ref{lem 4.5} that there exists $u_0 \in X$ such that,
after passing to a subsequence,
\begin{equation*}
   u_{n} \rightarrow u_0 \quad \text{in} \ X \ \text{as}\ n \rightarrow \infty.
\end{equation*}
Consequently,  $u_0 \in \mathcal{K}$ and
\begin{equation*}
  I (u_0) = \lim_{n \to \infty} I({u}_{n})=c_{g}.
\end{equation*}
This completes the proof.
\end{proof}

\section{Proof of Theorem \ref{th 1.3}} \label{sec 5}

\indent

In this section, we will give the proof of Theorem \ref{th 1.3} on the minimax characterization
of ground state solutions for \eqref{eq 1.6} in the case where $3 \leq p <4$. In the following,
we always assume that $b \geq 0$, $p \geq 3$, and that $(V_0)$, $(V_2)$ and $(V_3)$ hold.
We start with some elementary observations.

\begin{lemma}\label{lem 5.1}
Suppose that $(V_0)$ and $(V_2)$ hold. Then we have
\begin{equation}\label{eq 5.1}
  (\nabla V(x),x) \geq 0 \qquad \text{for all $x\in \R^2$.}
\end{equation}
\end{lemma}

\begin{proof}
For any fixed $x\in \R^2$,  define $f: (0, \infty) \rightarrow \R$ by
\begin{equation*}
  f(t) = t^2 V(x) - t^2 V(t^{-1}x) + \frac{1-t^2}{2}(\nabla V(x),x).
\end{equation*}
With an easy computation, we find that
\begin{equation*}
  f'(t) = 2t \left( \mathcal{V}(x) - \mathcal{V}(t^{-1}x)\right) \qquad \text{for $t>0$}.
\end{equation*}
It then follows from  $(V_2)$ that $f'(t) \leq 0$ on $(0, 1)$ and $f'(t) \geq 0$ on $(1,\infty)$.
This implies that
\begin{equation*}
  f(t) \geq f(1)=0 \qquad \text{for $t>0$,}
\end{equation*}
so that
\begin{equation}\label{eq 5.2}
  t^2 V(x) + \frac{1-t^2}{2}(\nabla V(x),x) \geq  t^2 V(t^{-1}x) \geq 0  \qquad \text{for $t>0$}
\end{equation}
in view of $(V_0)$. By passing to the limit $t \rightarrow 0^+$ in \eqref{eq 5.2}, we arrive at \eqref{eq 5.1},
as claimed.
\end{proof}

\begin{lemma}\label{lem 5.2}
Suppose that $\beta: (0, \infty) \to \R$ is a $C^1$-function, and that $t \in (0, \infty) \mapsto t^{-1} \beta'(t)$
is a bounded nonincreasing function with a positive lower bound.
Let $C_{i} \in \R$ for $i=1,2,3,4$, and let $C_1, C_3 >0$ and $C_4 \geq 0$. If $p \geq 3$, then the function
\begin{equation*}
  g:(0,\infty) \to \R,\qquad g(t) = C_{1}\beta(t) + C_{2}t^{4} - C_{3}t^{4}\log t-C_{4}t^{2p-2}
\end{equation*}
has a unique positive critical point $t_0$ such that $g'(t)>0$ for $t<t_0$ and $g'(t)<0$ for $t>t_0$.
\end{lemma}
\begin{proof}
The proof is elementary, so we omit it.
\end{proof}

Similarly as in \cite{Du-Weth-2017}, we now consider the Nehari-Pohozaev mainfold
$\mathcal{M}$ defined in \eqref{eq 1.10}, i.e.,
\begin{equation*}
    \mathcal{M}=\left\{u \in X \backslash \{0\}: \: J(u)=0\right\},
\end{equation*}
where $J: X \to \R$ is defined in \eqref{eq 1.9}. It is easy to see that
\begin{equation*}
    J(u)=2I'(u)u-P(u),
\end{equation*}
where $P(u)$ is given in Lemma~\ref{lem 2.4}. As already noted in the introduction, by Lemma \ref{lem 2.4}
we obtain that every nontrivial critical point of $I$ is contained in $\mathcal{M}$.
If $u\in \mathcal{M}$, then
\begin{equation}\label{eq 5.3}
  I(u) = \frac{1}{4}\int_{\R^2} \Bigl[V(x)+\frac{1}{2}(\nabla V(x),x)\Bigr]u^2\,\textrm{d}x
    +\frac{1}{32 \pi}|u|_2^4 + \frac{b(p-3)}{2p}|u|^{p}_{p},
\end{equation}
and since $p \geq 3$, by $(V_0)$ and Lemma \ref{lem 5.1} we have
\begin{equation}\label{eq 5.4}
  I(u) \geq \frac{1}{32 \pi} |u|_2^4 >0 \qquad \text{for $u \in \mathcal{M}$}.
\end{equation}
With a slight abuse of notation, we define
\begin{equation*}
  m = \inf_{u \in \mathcal{M}}I(u),
\end{equation*}
and we will prove that $m$ is attained by some $u \in \mathcal{M}$ which is
a critical point of $I$ in $X$, and hence a ground state solution of \eqref{eq 1.6}.

For $u\in X$ and $t>0$, we define the rescaled function $Q(t,u) \in X$ by $Q(t,u) = u_t$, i.e.,
\begin{equation*}
  Q(t,u)(x)= u_t(x)=t^{2}u(tx) \qquad \text{for $x \in \R^2$.}
\end{equation*}
In the following, we state some basic properties of $\mathcal{M}$ and $I$.

\begin{lemma}\label{lem 5.3}
Let $u\in X \setminus\{0\}$, then there exists a unique $t_u \in (0, \infty)$ such that
$Q(t_{u}, u) \in \mathcal{M}$ and $I(Q(t_{u},u))= \max_{t>0} I(Q(t,u))$.
Moreover, the map $X \backslash \{0\} \to (0,\infty)$, $u \mapsto t_u$ is continuous.
\end{lemma}

\begin{proof}
For $u \in X \setminus \{0\}$, define the function $h_u: (0,\infty) \to \R$, $
h_u(t):=I(Q(t,u))$. As in \eqref{eq 4.3}, we find that
\begin{align*}
  h_u(t)&=\frac{t^{4}}{2}\int_{\mathbb{R}^{2}}|\nabla u|^{2} \,\textrm{d}x
    +\frac{t^{2}}{2}\int_{\mathbb{R}^{2}} V(t^{-1}x)u^{2} \,\textrm{d}x
    +\frac{t^{4}}{8 \pi}\int_{\mathbb{R}^{2}}\int_{\mathbb{R}^{2}}
    \log\left(|x-y|\right)u^{2}(x)u^{2}(y)\,\textrm{d}x\textrm{d}y \notag \\
  &\quad-\frac{t^{4}\log t}{8 \pi}\left(\int_{\mathbb{R}^{2}}| u|^{2}\textrm{d}x\right)^{2}
    -\frac{bt^{2 p-2}}{p}\int_{\mathbb{R}^{2}}|u|^{p}\,\textrm{d}x.
\end{align*}
Consider now the function $\beta: (0, \infty) \rightarrow \R$ defined by
\begin{equation*}
  \beta(t)= t^2 \int_{\mathbb{R}^{2}} V(t^{-1}x)u^{2} \,\textrm{d}x.
\end{equation*}
Using $(V_2)$, \eqref{eq 1.11} and Lemma \ref{lem 5.2}, $h_u$ has a unique critical point
$t_u>0$ such that
\begin{equation}\label{eq 5.5}
  h_u'(t)>0 \quad \text{for $t \in (0, t_u)$} \qquad \text{and}\qquad
  h_u'(t)<0 \quad \text{for $t \in (t_u, \infty)$.}
\end{equation}
Since
\begin{equation*}
  h_u'(t)=  \frac{J(Q(t,u))}{t} \qquad \text{for $t>0$,}
\end{equation*}
we then see that $\max_{t>0} h_u(t)$ is attained at a unique $t = t_u$ so that
$h_u'(t) = 0$ and $Q(t_u, u) \in \mathcal{M}$.
Combining \eqref{eq 5.5} with the fact that the map $X \setminus \{0\} \to \R, \:
u \mapsto h_u'(t)$ is continuous for fixed $t>0$, we also derive that the map
$X \backslash \{0\} \to (0,\infty)$, $u \mapsto t_u$ is continuous, as claimed.
\end{proof}

By Lemma \ref{lem 5.3}, we immediately have the following corollary.

\begin{corollary}\label{coro 5.4}
The Nehari-Pohozaev manifold $\mathcal{M}$ is not empty, and the infimum of $I$ on $\mathcal{M}$
obeys the following minimax characterization:
\begin{equation*}
  \inf_{u\in \mathcal{M}} I(u) = \inf_{u\in X\backslash \{0\}} \sup_{t>0} I(u_t).
\end{equation*}
\end{corollary}

Next, we give a general result which will be used later.
\begin{lemma}\label{lem 5.5}
Let $u \in X$. Then we have
\begin{equation*}
  I\left(Q(t, u)\right) \leq I(u) - \frac{1-t^4}{4} J(u)\qquad \text{for all $t>0$}.
\end{equation*}
\end{lemma}

\begin{proof}
For $u \in X$, consider the function $\varphi_u: (0, \infty) \rightarrow \R$ defined by
\begin{equation*}
  \varphi_u (t) = I(u) - I\left(Q(t, u)\right) - \frac{1-t^4}{4} J(u).
\end{equation*}
It is easy to verify that $\varphi_u (1) = 0$ and
\begin{equation*}
  \varphi'_u (t) = t^3\left[h'_u(1) - \frac{h_u'(t)}{t^3}\right] \qquad \text{for $t>0$.}
\end{equation*}
Combining this with the fact that the function $(0, \infty) \rightarrow \R$,  $ t \mapsto \frac{h_u'(t)}{t^3}$
is nonincreasing on $(0, \infty)$, we obtain
\begin{equation*}
  \varphi'_u (t) \leq 0 \quad \text{for $t \in (0, 1)$} \qquad \text{and} \qquad
  \varphi'_u (t) \geq 0 \quad \text{for $t \in (1,\infty)$}.
\end{equation*}
This implies that
\begin{equation*}
  \varphi_u(t) \geq \varphi_u(1)=0 \qquad \text{for $t>0$,}
\end{equation*}
and thus the claim follows.
\end{proof}

The following lemma shows that $m>0$.

\begin{lemma}\label{lem 5.6}
There results $m > 0$.
\end{lemma}

\begin{proof}
For any $u\in \mathcal{M}$, we can deduce from \eqref{eq 1.11}, \eqref{eq 2.2}
and the Sobolev inequalities that
\begin{equation*}
  2|\nabla u|_2^2 + V_0 |u|_2^2  \leq \frac{2b(p-1)}{p} |u|_p^p - N_0 (u) + \frac{1}{8 \pi} |u|_2^4
  \leq C_1\|u\|^p + C_2\|u\|^4.
\end{equation*}
Since $u \neq 0$ and $p>2$, it holds that
\begin{equation}\label{eq 5.6}
  \inf_{u\in \mathcal{M}} \|u\|^2 >0.
\end{equation}

Now according to \eqref{eq 5.4}, we have $m \geq 0$. If $m >0$ holds, then the lemma is proved.
Therefore, arguing by contradiction, we assume that $m=0$. Let $\{u_n\} \subset \mathcal{M}$
be a minimizing sequence for $I$, that is, $I(u_n) \rightarrow 0$ as $n \rightarrow \infty$.
By Lemma \ref{lem 4.3}, $\{u_n\}$ is bounded in $H^1(\R^2)$.
Moreover, from \eqref{eq 5.4} we see that $|u_n|_2 \rightarrow 0$ as $n \rightarrow \infty$.
It then follows from \eqref{eq 1.11}, \eqref{eq 2.2} and the Gagliardo-Nirenberg inequality that
\begin{equation*}
  \int_{\R^2} \mathcal{V}(x) u_n^2\, \textrm{d}x \rightarrow 0, \quad
  N_2(u_n) \rightarrow 0 \quad \text{and} \quad
  |u_n|_p \rightarrow 0 \qquad \text{as  $n \rightarrow \infty.$}
\end{equation*}
Since $\{u_n\} \subset \mathcal{M}$, we further obtain
\begin{equation*}
  |\nabla u_n|_2 \rightarrow 0 \quad \text{and} \quad
  N_1(u_n) \rightarrow 0 \qquad \text{as  $n \rightarrow \infty.$}
\end{equation*}
This implies that $\|u_n\| \rightarrow 0$ as $n \rightarrow \infty$,  contradicting \eqref{eq 5.6}.
The proof is thus finished.
\end{proof}

In the sequel, we wish to show that $m>0$ is achieved. To this end, we need to consider the associated
limit equation \eqref{eq 3.4}, the associated limit functional $I_\infty$ and the corresponding
Nehari-Pohozeav manifold
\begin{equation*}
  \mathcal{M}_\infty := \left\{ u \in X \backslash \{0\}: \: J_\infty (u) =0 \right\},
\end{equation*}
where the functional $J_\infty: X \to \R$ is defined by
\begin{align*}
  J_\infty(u)=&\int_{\R^2} \left(2 |\nabla u|^2+  V_\infty u^2 -\frac{2b(p-1)}{p}|u|^{p}\right) {\rm d}x
    - \frac{1}{8 \pi} \left(\int_{\mathbb{R}^{2}} u^{2}\,{\rm{d}}x\right)^{2} \nonumber \\
   &+ \frac{1}{2 \pi}\int_{\R^2}\int_{\R^2}\log\left(|x-y|\right)u^{2}(x)u^{2}(y)\,{\rm d}x{\rm d}y,
\end{align*}
Then, we set
\begin{equation*}
  m_\infty = \inf_{u \in \mathcal{M}_\infty}I_\infty(u).
\end{equation*}

Now we are going to compare $m$ and $m_\infty$, which is used to show that
every minimizing sequence for $m$ is bounded in $X$.
\begin{lemma}\label{lem 5.7}
We have $m < m_\infty$.
\end{lemma}

\begin{proof}
By \cite[Theorem 1.2]{Du-Weth-2017}, we know that $m_{\infty}>0$ is attained at
a positive ground state solution $w \in X$ of \eqref{eq 3.4}.
Using $(V_0)$ and $(V_2)$, there exists $x_0 \in \R^2$ such that
\begin{equation*}
  \mathcal{V}(x_0)=V(x_0)=V_0.
\end{equation*}
This, together with \eqref{eq 1.11}, implies that
\begin{equation*}
  \int_{\R^2} \mathcal{V}(x) w^2\:\textrm{d}x  < \int_{\R^2} V_\infty w^2\:\textrm{d}x,
\end{equation*}
and so
\begin{equation*}
  J(w) < J_\infty(w) =0.
\end{equation*}
It then follows from Lemma \ref{lem 5.3} that there exists a unique $t_w \in (0,1)$
such that $Q(t_w, w) \in \mathcal{M}$. Therefore, using \eqref{eq 5.3} and $(V_3)$ we obtain
\begin{align*}
  m&\leq I(Q(t_w,w)) \leq \frac{t_w^2}{4}V_\infty |w|_2^2 +\frac{t_w^4}{32 \pi}|w|_2^4
    + \frac{b(p-3)}{2p}t_w^{2p-2}|w|^{p}_{p}\\
   &< \frac{1}{4}V_\infty |w|_2^2 +\frac{1}{32 \pi}|w|_2^4+ \frac{b(p-3)}{2p}|w|^{p}_{p}\\
   &= I_\infty(w) - \frac{1}{4}J_\infty(w) = I_\infty(w) = m_\infty.
\end{align*}
This completes the proof.
\end{proof}

The following key Proposition is the final step in the proof of Theorem \ref{th 1.3}.
\begin{proposition}\label{prop 5.8}
The level $m$ is achieved, and every minimizer of $m$ is a critical point
of $I$ which does not change sign on $\R^2$.
\end{proposition}

\begin{proof}
In the following, we divide the proof into three parts.

(i) We first prove that $m$ can be attained. Let $\{u_n\} \subset \mathcal{M}$
be a minimizing sequence for $I$, i.e., $I(u_n) \rightarrow m$ as $n \rightarrow \infty$.
By Lemma \ref{lem 4.3}, $\{u_n\}$ is bounded in $H^1(\R^2)$. Then,
up to a subsequence, there exists $u \in H^1(\R^2)$ such that
\begin{equation*}
  u_n \rightharpoonup u \quad \text{in} \ H^{1}(\R^{2}), \quad
  u_n \rightarrow u \quad \text{in}\ L^s_{loc}(\R^{2})\ \text{for all}\ s \geq 1, \quad
  u_n(x)\rightarrow u(x) \quad \text{a.e. in}\ \R^2.
\end{equation*}

We now claim that $u \neq 0$. Arguing by contradiction, we may assume that $u=0$, which implies
in particular that $u_n \rightarrow 0$ in $L^2_{loc}(\R^{2})$.
Then, using the fact that $\{\|u_n\|\}$ is bounded, we derive from $(V_0)$ and \eqref{eq 1.11} that
\begin{equation}\label{eq 5.7}
  \lim_{n \rightarrow \infty} \int_{\mathbb{R}^2}\left|\mathcal{V}(x)-V_\infty\right|u^2_n\,\textrm{d}x=0.
\end{equation}
Moreover, it follows from \eqref{eq 1.11}, \eqref{eq 2.2}, Lemma \ref{lem 5.6} and
the Gagliardo-Nirenberg inequality that, after passing to a subsequence,
\begin{equation}\label{eq 5.8}
  \lim_{n \rightarrow \infty} |u_n|^2_2 := \kappa \in (0, \infty).
\end{equation}
Using \eqref{eq 1.11} again, we have
\begin{equation*}
  J_\infty(u_n) \geq J(u_n) = 0.
\end{equation*}
From \cite[Lemma 4.2]{Du-Weth-2017}, we infer that for each $n \in \mathbb{N}$,
there exists $t_n \geq 1$ such that $Q(t_n, u_n) \in \mathcal{M}_{\infty}$, that is,
\begin{equation}\label{eq 5.9}
  2 t_n^4 |\nabla u_n|_2^2 + t_n^2 V_\infty |u_n|_2^2 + t_n^4 N_0(u_n)
   = \frac{1}{2 \pi} t_n^4 \log t_n |u_n|_2^4 + \frac{1}{8 \pi}t_n^4 |u_n|_2^4
     + \frac{2b(p-1)}{p} t_n^{2p -2} |u_n|_p^p.
\end{equation}
By the Sobolev inequality, $\{|u_n|_p\}$ is bounded and, up to a subsequence, we then have
\begin{equation}\label{eq 5.10}
  \lim_{n \rightarrow \infty} |u_n|_p^p :=\sigma \in [0, \infty).
\end{equation}
Since $\{u_n\} \subset \mathcal{M}$,  we may use \eqref{eq 5.7}$-$\eqref{eq 5.9} to obtain
\begin{equation*}
  \left(t_n^{-2} - 1\right) \left(V_\infty \kappa + o(1)\right)
  =  \frac{1}{2 \pi} \log t_n  \left( \kappa^2 + o(1)\right)
   + \frac{2b(p-1)}{p} \left(t_n^{2p-6}-1 \right) \left(\sigma + o(1)\right)
\end{equation*}
for all $n \in \N$, which implies that $t_n  \rightarrow 1$ as $n\rightarrow \infty$.
Therefore, by \eqref{eq 5.3} and \eqref{eq 5.7} we have
\begin{align*}
  m_\infty &\leq I_\infty(Q(t_n, u_n)) = \frac{t_n^2}{4} V_\infty |u_n|_2^2
    + \frac{t_n^4}{32 \pi} |u_n|_2^4 + \frac{b(p-3)}{2p} t_n^{2p-2} |u_n|^{p}_{p} \\
 &= \frac{t_n^2}{4} \int_{\R^2} \mathcal{V}(x) u_n^2\,\textrm{d}x + \frac{t_n^2}{32 \pi} |u_n|_2^4
    + \frac{b(p-3)}{2p} t_n^{2} |u_n|^{p}_{p}+ o(1)\\
 &= t_n^2 I(u_n) + o(1).
\end{align*}
Passing to the limit, we then find that $ m_\infty \leq m$, contradicting Lemma \ref{lem 5.7}.
So $u \neq 0$, as claimed.

Since $\{u_n\} \subset \mathcal{M}$,  we have
\begin{equation*}
  B_1\left(u_n^2, u_n^2\right) = N_1(u_n) =  N_2(u_n) + \frac{2b(p-1)}{p}|u_n|_p^p + \frac{1}{8 \pi} |u_n|_2^4
  - 2|\nabla u_n|_2^2 - \int_{\R^2} \mathcal{V}(x)u_n^2\,\textrm{d}x,
\end{equation*}
which means that $\sup_{n \in \mathbb{N}} B_1\left(u_n^2, u_n^2\right) < \infty$ due to the boundedness
of $\{\|u_n\|\}$. Consequently, $|u_n|_\ast$ remains bounded in $n$ by Lemma \ref{lem 2.2}, so that $\{u_n\}$
is bounded in $X$.  Then, passing to a subsequence if necessary, we may assume that $u_{n} \rightharpoonup u$
in $X$, and so $u \in X$. It follows from Lemma \ref{lem 2.1}(i)
that $u_{n} \rightarrow u$ in $L^{s}(\mathbb{R}^{2})$ for $s \geq 2$. Using the weak lower semicontinuity
of the norm and Lemma \ref{lem 2.1}(iv), we thus conclude that
\begin{equation}\label{eq 5.11}
  I(u) \leq \liminf_{n \rightarrow \infty} I(u_n) = m,
\end{equation}
and
\begin{equation}\label{eq 5.12}
  J(u) \leq 0.
\end{equation}
If $J(u) = 0$, then $u \in \mathcal{M}$ and \eqref{eq 5.11} implies that $m$ is attained at $u$.
Since \eqref{eq 5.12} holds, we only need to treat the case where
\begin{equation}\label{eq 5.13}
  J(u) < 0.
\end{equation}
We now show that if \eqref{eq 5.13} occurs, we reach a contradiction. Indeed, Lemmas \ref{lem 5.3} and
\ref{lem 5.5} indicate that there exists a unique $t_u \in (0,1)$ such that $Q(t_u, u) \in \mathcal{M}$ and
\begin{equation*}
  I\left(Q(t_u, u)\right) < I(u) - \frac{1}{4} J(u),
\end{equation*}
so that
\begin{align*}
  m&\leq I\left(Q(t_u, u)\right) < I(u) - \frac{1}{4} J(u) \\
   &=\frac{1}{4}\int_{\R^2} \bigl[V(x)+\frac{1}{2}(\nabla V(x),x)\bigr]u^2\,\textrm{d}x
    +\frac{1}{32 \pi}|u|_2^4 + \frac{b(p-3)}{2p}|u|^{p}_{p}\\
   &\leq \lim_{n \rightarrow \infty}\left(\frac{1}{4}\int_{\R^2} \bigl[V(x)+\frac{1}{2}(\nabla V(x),x)\bigr]u_n^2\,\textrm{d}x
    +\frac{1}{32 \pi}|u_n|_2^4 + \frac{b(p-3)}{2p}|u_n|^{p}_{p}\right)\\
   &=\lim_{n\rightarrow \infty} I(u_n) = m.
\end{align*}
This is impossible, and part (i) is thus finished.

(ii) We next prove that every minimizer of $m$ is a critical point of $I$.
Let $u \in \mathcal M$ be an arbitrary minimizer for $I$ on $\mathcal M$.
We show that $I'(u) v =0$ for all $v \in X$, and thus $u$ is a critical point
of $I$. Suppose by contradiction that this is false. Then, there exists $v \in X$
such that $I'(u)v < 0$. Since $I$ is a $C^1$-functional on $X$, we can fix $\eps>0$ with
the following property: \medskip

{\em For every $w_i \in X$ with $\|w_i\|_{X}<\eps$, $i=1,2$, and every $\tau \in (0,\eps)$, we have}
\begin{equation*}
  I(u + w_1 + \tau (v + w_2)) \le I(u +w_1)-\eps \tau.
\end{equation*}
Using Lemma \ref{lem 5.3} and the fact that $t_u=1$, we may choose $\tau \in (0,\eps)$
sufficiently small such that for $t^\tau:= t_{u+\tau v}$,
\begin{equation*}
  \|Q(t^\tau,u)- u\|_X < \eps \qquad \text{and}\qquad  \|Q(t^\tau,v)- v\|_X < \eps.
\end{equation*}
Setting $w_1:= Q(t^\tau,u)-u$ and $w_2:= Q(t^\tau,v)-v$, we obtain that $\|w_i\|_{X}<\eps$
for $i=1,2$. It then follows from the above property that
\begin{align*}
  I\left(Q(t^\tau,u + \tau v)\right) &= I\left(Q(t^\tau,u) + \tau Q(t^\tau, v)\right)
   = I(u + w_1 + \tau (v+w_2)) \\
  &\le I(u+w_1) - \eps \tau < I(u+w_1) = I(Q(t^\tau,u)) \le I(u)= m.
\end{align*}
Since $Q(t^\tau,u + \tau v) \in \mathcal{M}$, this contradicts the definition of $m$.
Therefore, part (ii) follows.

(iii) We finally prove that every minimizer of $m$ does not change sign in $\R^2$.
If $u$ is a minimizer of $I|_{\mathcal M}$,
then $|u|$ is also a minimizer of $I|_{\mathcal M}$
due to the fact that $I(u)=I(|u|)$ and $J(u)=J(|u|)$.
Hence, $|u|$ is a critical point of $I$ by the considerations above.
Then the standard elliptic regularity theory gives that
$|u| \in C^{2}(\mathbb{R}^{2})$ and $-\Delta |u| + q(x) |u| = 0$ in $\mathbb{R}^{2}$
with some function $q(x) \in L_{loc}^{\infty}(\mathbb{R}^{2})$.
Consequently, the strong maximum principle and the fact that $u \neq 0$ imply that
$|u|>0$ in $\mathbb{R}^{2}$, which shows that $u$ does not change sign in $\R^2$.
The proof of this proposition is finished.
\end{proof}

The {\em proof of Theorem \ref{th 1.3}} is now completed by combining Proposition \ref{prop 5.8}
and Corollary \ref{coro 5.4}.

\vskip 0.6 true cm

\end{document}